\documentclass[twoside,english,american,DIV12,listtotoc,bibtotoc,idxtotoc]{scrartcl}
\usepackage[latin1]{inputenc}
\usepackage{verbatim}
\usepackage{amsthm}
\usepackage{amsmath}
\usepackage{amssymb}

\makeatletter
\theoremstyle{plain}
\newtheorem{thm}{\protect\theoremname}[section]
  \theoremstyle{definition}
  \newtheorem{defn}[thm]{\protect\definitionname}
  \theoremstyle{definition}
  \newtheorem{example}[thm]{\protect\examplename}
  \theoremstyle{remark}
  \newtheorem{rem}[thm]{\protect\remarkname}
  \theoremstyle{plain}
  \newtheorem{lem}[thm]{\protect\lemmaname}

\usepackage{helvet}
\usepackage[T1]{fontenc}

\usepackage{amsmath}
\usepackage{setspace}
\usepackage{amssymb}
\usepackage{enumerate}
\usepackage{url}

\makeatletter

\usepackage[T1]{fontenc}    

\usepackage{a4wide}        
\usepackage{tu-preprint}
\usepackage{amsmath}

\usepackage{euscript}
\usepackage{mathtools}
\allowdisplaybreaks[1] 
\usepackage{amsfonts}
\newenvironment{keywords}{ \noindent\footnotesize\textbf{Keywords and phrases:}}{}

\newenvironment{class}{\noindent\footnotesize\textbf{Mathematics subject classification 2010:}}{}

\usepackage{color,tu-preprint}

\newcommand*{\dive}{\operatorname{div}}

\newcommand*{\curl}{\operatorname{curl}}

\newcommand*{\grad}{\operatorname{grad}}

\DeclareMathAccent{\Circ}{\mathalpha}{operators}{"17}

\renewcommand{\Re}{\operatorname{\mathfrak{Re}}}

\renewcommand{\tilde}{\widetilde}
\renewcommand*{\epsilon}{\varepsilon}
\renewcommand*{\rho}{\varrho}

\makeatother

\usepackage{babel}
\institut{Institut f\"ur Analysis}

\preprintnumber{MATH-AN-03-2013}

\preprinttitle{G-convergence of linear differential equations.}

\author{Marcus Waurick}

\AtBeginDocument{
  
}

\makeatother

\usepackage{babel}
  \addto\captionsamerican{\renewcommand{\definitionname}{Definition}}
  \addto\captionsamerican{\renewcommand{\examplename}{Example}}
  \addto\captionsamerican{\renewcommand{\lemmaname}{Lemma}}
  \addto\captionsamerican{\renewcommand{\remarkname}{Remark}}
  \addto\captionsamerican{\renewcommand{\theoremname}{Theorem}}
  \addto\captionsenglish{\renewcommand{\definitionname}{Definition}}
  \addto\captionsenglish{\renewcommand{\examplename}{Example}}
  \addto\captionsenglish{\renewcommand{\lemmaname}{Lemma}}
  \addto\captionsenglish{\renewcommand{\remarkname}{Remark}}
  \addto\captionsenglish{\renewcommand{\theoremname}{Theorem}}
  \providecommand{\definitionname}{Definition}
  \providecommand{\examplename}{Example}
  \providecommand{\lemmaname}{Lemma}
  \providecommand{\remarkname}{Remark}
\providecommand{\theoremname}{Theorem}

\begin{document}
\makepreprinttitlepage

\author{ Marcus Waurick \\ Institut f\"ur Analysis, Fachrichtung Mathematik\\ Technische Universit\"at Dresden\\ Germany\\  marcus.waurick@tu-dresden.de }

\title{G-convergence of linear differential equations.}

\maketitle
\begin{abstract} \textbf{Abstract.} We discuss $G$-convergence of
linear integro-differential-algebaric equations in Hilbert spaces.
We show under which assumptions it is generic for the limit equation
to exhibit memory effects. Moreover, we investigate which classes
of equations are closed under the process of $G$-convergence. The
results have applications to the theory of homogenization. As an example
we treat Maxwell's equation with the Drude-Born-Fedorov constitutive
relation.\end{abstract}


\begin{class} 34L99 (Ordinary differential operators), 34E13 (Ordinary
differential equations, multiple scale methods), 34A08 (Fractional
differential equations), 35B27 (Partial differential equations, homogenization;
equations in media with periodic structure), 35Q61 (Maxwell's equations),
45A05 (Linear integral equations)  \end{class}

\newpage

\tableofcontents{} 

\newpage

\section{Introduction}

We discuss some issues occuring in the homogenization of linear integro-differential
equations in Hilbert spaces. Similar to \cite{Piccinini1978,Piccinini77,Tartar_Mem,TartarNonlHom}
we understand homogenization theory as the study of limits of sequences
of equations in the sense of $G$-convergence. Whereas in \cite{Piccinini1978,Piccinini77}
(non-linear) ordinary differential equations in finite-dimensional
space are considered, we choose the perspective given in \cite{Mascarenhas1984,Petrini1998,Tartar_Mem,TartarNonlHom,Anton,Jiang,Jiang05}.
 The abstract setting is the following.
\begin{defn}[{$G$-convergence, {\cite[p.\ 74]{Gcon1}, \cite{Waurick2012aDIE}}}]
 Let $H$ be a Hilbert space. Let $(A_{n}:D(A_{n})\subseteqq H\to H)_{n}$
be a sequence of continuously invertible linear operators onto $H$
and let $B:D(B)\subseteqq H\to H$ be linear and one-to-one. We say
that $(A_{n})_{n}$ \emph{$G$-converges to $B$} if $\left(A_{n}^{-1}\right)_{n}$
converges in the weak operator topology to $B^{-1}$, i.e., for all
$f\in H$ the sequence $(A_{n}^{-1}(f))_{n}$ converges weakly to
some $u$, which satisfies $u\in D(B)$ and $B(u)=f$. $B$ is called
the%
\footnote{Note that the $G$-limit is uniquely determined, cf. \cite[Proposition 4.1]{Waurick2012aDIE}.%
} \emph{$G$-limit} of $(A_{n})_{n}$ and we write $A_{n}\stackrel{G}{\longrightarrow}B$. 
\end{defn}
Our starting point will be equations of the form
\[
\partial_{0}\mathcal{M}u+\mathcal{N}u=f,
\]
where $\mathcal{M},\mathcal{N}$ are suitable operators in space-time
and $\partial_{0}$ is the time-derivative established in a Hilbert
space setting to be specified below (see also \cite{Picard,Kalauch}).
In the usual framework of homogenization theory, one assumes $\mathcal{M}$
and $\mathcal{N}$ to be multiplication operators in space-time, i.e.,
there are mappings $a$ and $b$ such that $\mathcal{M}=a(\cdot)$
and $\mathcal{N}=b(\cdot)$. Assuming well-posedness of the above
equation, i.e., existence, uniqueness and continuous dependence on
the right-hand side $f$ in a suitable (Hilbert space) framework,
one is interested in the sequence of equations 
\begin{equation}
\partial_{0}\mathcal{M}_{n}u_{n}+\mathcal{N}_{n}u_{n}=f\label{eq:intr_hom}
\end{equation}
with $\mathcal{M}_{n}=a(n\cdot)$ and correspondingly for $\mathcal{N}_{n}$
yielding a sequence of solutions $(u_{n})_{n}$. The question arises,
whether the sequence $(u_{n})_{n}$ converges and if so whether the
respective limit $u$ satisfies an equation of similar form. A formal
computation in (\ref{eq:intr_hom}) reveals that 
\[
u_{n}=\left(\partial_{0}\mathcal{M}_{n}+\mathcal{N}_{n}\right)^{-1}f.
\]
Thus, if we show the convergence of $\left(\partial_{0}\mathcal{M}_{n}+\mathcal{N}_{n}\right)^{-1}$
in the weak opertor topology to some one-to-one mapping $C\eqqcolon B^{-1}$,
we deduce the weak convergence of $(u_{n})_{n}$ the limit of which
denoted by $u$ satisfies
\[
Bu=f.
\]
In other words, $\left(\partial_{0}\mathcal{M}_{n}+\mathcal{N}_{n}\right)$
$G$-converges to $B$.

In this article we think of $\left(\mathcal{M}_{n}\right)_{n}$ and
$\left(\mathcal{N}_{n}\right)_{n}$ to be bounded sequences of bounded
linear operators in space-time. We want to discuss assumptions on
these sequences guaranteeing a compactness result with respect to
$G$-convergence. Moreover, we outline possible assumptions yielding
the closedness under $G$-convergence and give examples for equations,
where the associated sequences of differential operators itself are
$G$-convergent. We exemplify our findings with examples from the
literature \cite{Mascarenhas1984,Tartar_Mem,TartarNonlHom,Jiang,Jiang05},
highlight possible connections and give an example for a Drude-Born-Fedorov
model in electro-magentism (see \cite{Picard-Freymond} and Example
\ref{Drude-Born-Fedorov-Ex} below), where homogenization theorems
are -- to the best of the author's knowledge -- not yet available
in the literature. We will also underscore the reason of the limit
equation to exhibit memory effects. An heuristic explanation is the
lack of continuity of computing the inverse with respect to the weak
operator topology.

In Section \ref{sec:Setting-and-main} we introduce the functional
analytic setting used for discussing integro-differential-algebraic
equations and state our main Theorems. We successively apply the results
from Section \ref{sec:Setting-and-main} to time-independent coefficients
(Section \ref{sec:Time-independent-coefficients}), time-translation
invariant coefficients (Section \ref{sec:Time-translation-invariant-coeff})
and time-dependent coefficients (Section \ref{sec:Time-dependent-coefficients}).
In each of the Sections \ref{sec:Time-independent-coefficients},
\ref{sec:Time-translation-invariant-coeff} and \ref{sec:Time-dependent-coefficients}
we give examples and discuss whether particular classes of equations
are closed under limits with respect to $G$-convergence. In Section
\ref{sec:Proof-of-the} we prove the main theorems of Section \ref{sec:Setting-and-main}.
The respective proofs rely on elementary Hilbert space theory.

\section{Setting and main theorems\label{sec:Setting-and-main}}

The key fact giving way for computations is the possibility of establishing
the time-derivative as a continuously invertible normal operator in
an exponentially weighted Hilbert space. %
{} For $\nu>0$ we define the operator 
\[
\partial_{0}\colon H_{\nu,1}(\mathbb{R})\subseteqq L_{\nu}^{2}(\mathbb{R})\to L_{\nu}^{2}(\mathbb{R}),f\mapsto f',
\]
where $L_{\nu}^{2}(\mathbb{R})\coloneqq L^{2}(\mathbb{R},\exp(-2\nu\cdot)\lambda)$
is the space of square-integrable functions with respect to the weighted
Lebesgue measure $\exp(-2\nu\cdot)\lambda$ and $H_{\nu,1}(\mathbb{R})$
is the space of $L_{\nu}^{2}(\mathbb{R})$-functions with distributional
derivative in $L_{\nu}^{2}(\mathbb{R})$. We denote the scalar-product
on $L_{\text{\ensuremath{\nu}}}^{2}(\mathbb{R})$ by $\langle\cdot,\cdot\rangle_{\nu}$
and the induced norm by $\left|\cdot\right|_{\nu}$. Of course the
operator $\partial_{0}$ depends on the scalar $\nu$. However, since
it will be obvious from the context, which value of $\nu$ is chosen,
we will omit the explicit reference to it in the notation of $\partial_{0}.$
It can be shown that $\partial_{0}$ is continuously invertible (\cite[Example 2.3]{Picard}
or \cite[Corollary 2.5]{Kalauch}). The norm bound of the inverse
is $1/\nu$. Of course the latter construction can be extended to
the Hilbert-space-valued case of $L_{\nu}^{2}(\mathbb{R};H)$-functions%
\footnote{We will also use the notation $\langle\cdot,\cdot\rangle_{\nu}$ and
$\left|\cdot\right|_{\nu}$ for the scalar product and norm in $L_{\nu}^{2}(\mathbb{R};H)$,
respectively.%
}. We will use the same notation for the time-derivative. In order
to formulate our main theorems related to the theory of homogenization
of ordinary differential equations, we need to introduce the following
notion.
\begin{defn}
Let $H_{0},H_{1}$ be Hilbert spaces, $\nu_{0}>0.$ We call a linear
mapping 
\begin{equation}
M\colon D(M)\subseteqq\bigcap_{\nu>0}L_{\nu}^{2}(\mathbb{R};H_{0})\to\bigcap_{\nu\geqq\nu_{0}}L_{\nu}^{2}(\mathbb{R};H_{1})\label{eq:evolutionary_map}
\end{equation}
\emph{evolutionary (at $\nu_{1}>0$)}%
\footnote{The notion ``evolutionary'' is inspired by the considerations in
\cite[Definition 3.1.14, p. 91]{Picard}, where polynomial expressions
in partial differential operators are considered.%
} if $D(M)$ is dense in $L_{\nu}^{2}(\mathbb{R};H)$ for all $\nu\geqq\nu_{1}$,
if $M$ extends to a bounded linear operator from $L_{\nu}^{2}(\mathbb{R};H_{0})$
to $L_{\nu}^{2}(\mathbb{R};H_{1})$ for all $\nu\geqq\nu_{1}$ and
is such that%
\footnote{For a linear operator $A$ from $L_{\nu}^{2}(\mathbb{R};H_{0})$ to
$L_{\nu}^{2}(\mathbb{R};H_{1})$ we denote its operator norm by $\left\Vert A\right\Vert _{L(L_{\nu}^{2}(\mathbb{R};H_{0}),L_{\nu}^{2}(\mathbb{R};H_{1}))}$.
If the spaces $H_{0}$ and $H_{1}$ are clear from the context, we
shortly write $\left\Vert A\right\Vert _{L(L_{\nu}^{2})}$. %
} 
\[
\limsup_{\nu\to\infty}\left\Vert M\right\Vert _{L(L_{\nu}^{2}(\mathbb{R};H_{0}),L_{\nu}^{2}(\mathbb{R};H_{1}))}<\infty.
\]
 The continuous extension of $M$ to some $L_{\nu}^{2}$ will also
be denoted by $M$. In particular, we will not distinguish notationally
between the different realizations of $M$ as a bounded linear operator
for different $\nu$ as these realizations coincide on a dense subset.
We define the set
\[
L_{\text{ev},\nu_{1}}(H_{0},H_{1})\coloneqq\{M;M\text{ is as in (\ref{eq:evolutionary_map}) and is evolutionary at }\nu_{1}\}.
\]
We abbreviate $L_{\text{ev},\nu_{1}}(H_{0})\coloneqq L_{\text{ev},\nu_{1}}(H_{0},H_{0})$.
A subset $\mathfrak{M}\subseteqq L_{\text{ev},\nu_{1}}(H_{0},H_{1})$
is called \emph{bounded} if $\limsup_{\nu\to\infty}\sup_{M\in\mathfrak{M}}\|M\|_{L(L_{\nu}^{2})}<\infty.$
A family $(M_{\iota})_{\iota\in I}$ in $L_{\text{ev},\nu_{1}}(H_{0},H_{1})$
is called \emph{bounded }if $\{M_{\iota};\iota\in I\}$ is bounded. 
\end{defn}
Note that $L_{\text{ev},\nu_{1}}(H_{0},H_{1})\subseteqq L_{\text{ev},\nu_{2}}(H_{0},H_{1})$
for all $\nu_{1}\leqq\nu_{2}.$ We give some examples of evolutionary
mappings.
\begin{example}
\label{Ex:time-in}Let $H$ be a Hilbert space and $M_{0}\in L(H)$.
Then there is a canonical extension $M$ of $M_{0}$ to $L_{\nu}^{2}(\mathbb{R};H)$-functions
such that $\left(M\phi\right)(t)\coloneqq\left(M_{0}\phi(t)\right)$
for all $\phi\in L_{\nu}^{2}(\mathbb{R};H)$ and a.e.\ $t\in\mathbb{R}$.
In that way $M\in\bigcap_{\nu>0}L_{\text{ev},\nu}(H)$. Henceforth,
we shall not distinguish notationally between $M$ and $M_{0}$.
\end{example}
~
\begin{example}
\label{time-dep}Let $H$ be a Hilbert space and let $L_{s}^{\infty}(\mathbb{R};L(H))$
be the space of bounded strongly measurable functions from $\mathbb{R}$
to $L(H)$. For $A\in L_{s}^{\infty}(\mathbb{R};L(H))$ we denote
the associated multiplication operator on $L_{\nu}^{2}(\mathbb{R};H)$
by $A(m_{0})$. Thus, also in this case, $A(m_{0})\in\bigcap_{\nu>0}L_{\text{ev},\nu}(H)$.
\end{example}
~
\begin{example}
\label{time-inv}For $\nu_{0}>0$ let $g\in L_{\nu_{0}}^{1}(\mathbb{R}_{>0})\coloneqq\{g\in L_{\text{loc}}^{1}(\mathbb{R});g=0\mbox{\text{ on }\ensuremath{\mathbb{R}_{<0}}},\int_{\mathbb{R}}|g(t)|e^{-\nu t}dt<\infty\}.$
By Young's inequality or by Example \ref{Ex:conv} below, we deduce
that $g*\in L_{\text{ev},\nu_{0}}(\mathbb{C}),$ where $g*f$ denotes
the convolution of some function $f$ with $g$.
\end{example}
To formulate our main theorems, we denote the weak operator topology
by $\tau_{\text{w}}$. Convergence within this topology is denoted
by $\stackrel{\tau_{\textnormal{w}}}{\to}$. Limits within this topology
are written as $\tau_{\textnormal{w}}\textnormal{-}\lim$. We will
extensively use the fact that for a separable Hilbert space $H$ bounded
subsets of $L(H)$, which are $\tau_{\textnormal{w}}$-closed, are
$\tau_{\textnormal{w}}$-sequentially compact. Our main theorems concerning
the $G$-convergence of differential equations read as follows.
\begin{thm}
\label{thm:first-hom_thm} Let $H$ be a separable Hilbert space,
$\nu_{0}>0$. Let $\left(\mathcal{M}_{n}\right)_{n}$, $\left(\mathcal{N}_{n}\right)_{n}$
be bounded sequences in $L_{\textnormal{ev},\nu_{0}}(H)$. Assume
there exists $c>0$ such that for all $n\in\mathbb{N}$ and $\nu\geqq\nu_{0}$
\[
\Re\langle\mathcal{M}_{n}\phi,\phi\rangle_{\nu}\geqq c\langle\phi,\phi\rangle_{\nu}\quad(\phi\in L_{\nu}^{2}(\mathbb{R};H)).
\]
Then there exists $\nu\geqq\nu_{0}$ and a subsequence $(n_{k})_{k}$
of $(n)_{n}$ such that
\[
\partial_{0}\mathcal{M}_{n_{k}}+\mathcal{N}_{n_{k}}\stackrel{G}{\longrightarrow}\partial_{0}\mathcal{M}_{hom,0}^{-1}+\partial_{0}\sum_{j=1}^{\infty}\left(-\sum_{\ell=1}^{\infty}\mathcal{M}_{hom,0}^{-1}\mathcal{M}_{hom,\ell}\right)^{j}\mathcal{M}_{hom,0}^{-1},
\]
as $k\to\infty$ in $L_{\nu}^{2}(\mathbb{R};H)$, where 
\[
\mathcal{M}_{hom,0}=\tau_{\textnormal{w}}\textnormal{-}\lim_{k\to\infty}\mathcal{M}_{n_{k}}^{-1}
\]
and 
\[
\mathcal{M}_{hom,\ell}=\tau_{\textnormal{w}}\textnormal{-}\lim_{k\to\infty}\mathcal{M}_{n}^{-1}\left(-\partial_{0}^{-1}\mathcal{N}_{n}\mathcal{M}_{n}^{-1}\right)^{\ell}.
\]
\end{thm}
\begin{rem}
(a) It should be noted that the positive-definiteness condition in
Theorem \ref{thm:second-hom_thm} is a well-posedness condition, i.e.,
a condition for $\partial_{0}\mathcal{M}_{n_{k}}+\mathcal{N}_{n_{k}}$
to be continuously invertible for all $\nu$ sufficiently large. Indeed,
for $f\in L_{\nu}^{2}(\mathbb{R};H)$ and $u\in L_{\nu}^{2}(\mathbb{R};H)$
with 
\[
\left(\partial_{0}\mathcal{M}_{n}+\mathcal{N}_{n}\right)u=f
\]
 we multiply by $\partial_{0}^{-1}$ and get 
\[
\left(\mathcal{M}_{n}+\partial_{0}^{-1}\mathcal{N}_{n}\right)u=\partial_{0}^{-1}f.
\]
 The positive definiteness condition yields, see also Lemma \ref{lem:pos-def},
the invertibility of $\mathcal{M}_{n}$. Hence, we arrive at
\[
\left(1+\mathcal{M}_{n}^{-1}\partial_{0}^{-1}\mathcal{N}_{n}\right)u=\mathcal{M}_{n}^{-1}\partial_{0}^{-1}f.
\]
Choosing $\nu>0$ sufficiently large, we deduce that the operator
$\left(1+\mathcal{M}_{n}^{-1}\partial_{0}^{-1}\mathcal{N}_{n}\right)$
is continuously invertible with a Neumann series expression.

(b) If $\mathcal{N}=0$ in Theorem \ref{thm:first-hom_thm}, then
we deduce that equations of the form $\partial_{0}\mathcal{M}u=f$
are closed under the process of $G$-convergence. If $\mathcal{N}\neq0$,
then the above theorem suggests that this is not true for equations
of the form $\left(\partial_{0}\mathcal{M}+\mathcal{N}\right)u=f$.
However, if we consider $\partial_{0}\mathcal{M}+\mathcal{N}$ as
$\partial_{0}\left(\mathcal{M}+\partial_{0}^{-1}\mathcal{N}\right)$,
the equations under consideration in Theorem \ref{thm:first-hom_thm}
are closed under $G$-limits. Indeed, the limit may be represented
by 
\[
\partial_{0}\left(\mathcal{M}_{hom,0}^{-1}+\sum_{k=1}^{\infty}\left(-\sum_{\ell=1}^{\infty}\mathcal{M}_{hom,0}^{-1}\mathcal{M}_{hom,\ell}\right)^{k}\mathcal{M}_{hom,0}^{-1}\right).
\]
 In the forthcoming sections we will further elaborate the aspect
of closedness under $G$-limits.
\end{rem}
In system or control theory one is interested in differential-algebraic
systems, see e.g.\ \cite{JacZwart1}. We, thus, formulate the analogous
statement for (integro-differential-)algebraic systems.
\begin{thm}
\label{thm:second-hom_thm}Let $H_{0},H_{1}$ be separable Hilbert
spaces, $\nu_{0}>0$. Let $\left(\mathcal{M}_{n}\right)_{n},\left(\mathcal{N}_{n}^{ij}\right)_{n}$
be bounded sequences in $L_{\textnormal{ev},\nu_{0}}(H_{0})$ and
$L_{\textnormal{ev},\nu_{0}}(H_{j},H_{i})$, respectively ($i,j\in\{0,1\}$).
Assume there exists $c>0$ such that for all $n\in\mathbb{N}$ and
$\nu\geqq\nu_{0}$ we have for all $(\phi,\psi)\in L_{\nu}^{2}(\mathbb{R};H_{0}\oplus H_{1})$
\[
\Re\langle\mathcal{M}_{n}\phi,\phi\rangle_{\nu}\geqq c\left|\phi\right|_{\nu}^{2},\quad\Re\langle\mathcal{N}_{n}^{11}\psi,\psi\rangle_{\nu}\geqq c\left|\psi\right|_{\nu}^{2}.
\]
Then there exists $\nu\geqq\nu_{0}$ and a subsequence $(n_{k})_{k}$
of $(n)_{n}$ such that
\begin{align*}
 & \partial_{0}\left(\begin{array}{cc}
\mathcal{M}_{n_{k}} & 0\\
0 & 0
\end{array}\right)+\left(\begin{array}{cc}
\mathcal{N}_{n_{k}}^{00} & \mathcal{N}_{n_{k}}^{01}\\
\mathcal{N}_{n_{k}}^{10} & \mathcal{N}_{n_{k}}^{11}
\end{array}\right)\\
 & \stackrel{G}{\longrightarrow}\left(\begin{array}{cc}
\partial_{0} & 0\\
0 & 1
\end{array}\right)\left(\left(\begin{array}{cc}
\mathcal{M}_{hom,0,00}^{-1} & 0\\
0 & \mathcal{N}_{hom,-1,11}^{-1}
\end{array}\right)^{\phantom{\ell}}\right.\\
 & \quad\quad\left.+\sum_{\ell=1}^{\infty}\left(-\left(\begin{array}{cc}
\mathcal{M}_{hom,0,00}^{-1} & 0\\
0 & \mathcal{N}_{hom,-1,11}^{-1}
\end{array}\right)\mathcal{M}^{(1)}\right)^{\ell}\left(\begin{array}{cc}
\mathcal{M}_{hom,0,00}^{-1} & 0\\
0 & \mathcal{N}_{hom,-1,11}^{-1}
\end{array}\right)\right),
\end{align*}
where we put 
\[
\mathcal{N}_{n}\coloneqq\mathcal{N}_{n}^{00}-\mathcal{N}_{n}^{01}\left(\mathcal{N}_{n}^{11}\right)^{-1}\mathcal{N}_{n}^{10}\quad(n\in\mathbb{N})
\]
as well as 
\[
\mathcal{M}^{(1)}\coloneqq\left(\begin{array}{cc}
\sum_{\ell=1}^{\infty}\mathcal{M}_{hom,\ell,00} & \sum_{\ell=0}^{\infty}\mathcal{M}_{hom,\ell,01}\\
\sum_{\ell=0}^{\infty}\mathcal{M}_{hom,\ell,10} & \sum_{\ell=0}^{\infty}\mathcal{M}_{hom,\ell,11}
\end{array}\right)
\]
and 
\begin{align*}
\mathcal{M}_{hom,\ell,00} & =\tau_{\textnormal{w}}\text{-}\lim_{k\to\infty}\mathcal{M}_{n_{k}}^{-1}\left(-\partial_{0}^{-1}\mathcal{N}_{n_{k}}\mathcal{M}_{n_{k}}^{-1}\right)^{\ell},\\
\mathcal{M}_{hom,\ell,01} & =\tau_{\textnormal{w}}\text{-}\lim_{k\to\infty}-\mathcal{M}_{n_{k}}^{-1}\left(-\partial_{0}^{-1}\mathcal{N}_{n_{k}}\mathcal{M}_{n_{k}}^{-1}\right)^{\ell}\partial_{0}^{-1}\mathcal{N}_{n_{k}}^{01}\left(\mathcal{N}_{n_{k}}^{11}\right)^{-1},\\
\mathcal{M}_{hom,\ell,10} & =\tau_{\textnormal{w}}\text{-}\lim_{k\to\infty}-\left(\mathcal{N}_{n_{k}}^{11}\right)^{-1}\mathcal{N}_{n_{k}}^{10}\mathcal{M}_{n_{k}}^{-1}\left(-\partial_{0}^{-1}\mathcal{N}_{n_{k}}\mathcal{M}_{n_{k}}^{-1}\right)^{\ell},\\
\mathcal{M}_{hom,\ell,11} & =\tau_{\textnormal{w}}\text{-}\lim_{k\to\infty}\left(\mathcal{N}_{n_{k}}^{11}\right)^{-1}\mathcal{N}_{n_{k}}^{10}\mathcal{M}_{n_{k}}^{-1}\left(-\partial_{0}^{-1}\mathcal{N}_{n_{k}}\mathcal{M}_{n_{k}}^{-1}\right)^{\ell}\partial_{0}^{-1}\mathcal{N}_{n_{k}}^{01}\left(\mathcal{N}_{n_{k}}^{11}\right)^{-1},\\
\mathcal{N}_{hom,-1,11} & =\tau_{\textnormal{w}}\text{-}\lim_{k\to\infty}\left(\mathcal{N}_{n_{k}}^{11}\right)^{-1}.
\end{align*}
\end{thm}
\begin{rem}
(a) As in Theorem \ref{thm:first-hom_thm} the positive definiteness
conditions in Theorem \ref{thm:second-hom_thm} serve as well-posed
conditions for the respective (integro-differential-algebraic) equations.
We will compute the respetive inverse in the proof of Theorem \ref{thm:second-hom_thm}.

(b) Note that, by the definition of $G$-convergence, both the Theorems
\ref{thm:first-hom_thm} and \ref{thm:second-hom_thm} implicitly
assert that the limit equations are well-posed, i.e., that the limit
operator is continuously invertible. In fact it will be the strategy
of the respective proofs to compute the limit of the respective solution
operators, which will be continuous linear operators and afterwards
inverting the limit.

(c) Assume that in Theorem \ref{thm:second-hom_thm} the expressions
$\mathcal{M}_{hom,\ell,00}$, \foreignlanguage{english}{$\mathcal{M}_{hom,\ell,01}$},
$\mathcal{M}_{hom,\ell,10}$, $\mathcal{M}_{hom,\ell,11},$ $\mathcal{N}_{hom,-1,11}$
can be computed without choosing subsequences. Then the sequence 
\[
\left(\partial_{0}\left(\begin{array}{cc}
\mathcal{M}_{n} & 0\\
0 & 0
\end{array}\right)+\left(\begin{array}{cc}
\mathcal{N}_{n}^{00} & \mathcal{N}_{n}^{01}\\
\mathcal{N}_{n}^{10} & \mathcal{N}_{n}^{11}
\end{array}\right)\right)_{n}
\]
 is $G$-convergent. Indeed, the latter follows with a subsequence
argument.

(c) Assuming $H_{1}=\{0\}$ and, as a consequence, $\mathcal{N}^{ij}=0$
for all $i,j\in\{0,1\}$ except $i=j=0$, we see that Theorem \ref{thm:second-hom_thm}
is more general than Theorem \ref{thm:first-hom_thm}. The generalization
in Theorem \ref{thm:second-hom_thm} is also needed in the theory
of homogenization of partial differential equations, see e.g.~\cite[Theorem 4.4]{Waurick2012aDIE}
for a more restrictive case. We give an example in the forthcoming
sections.
\end{rem}
For convenience, we include easy examples that show that the assumptions
in the above theorems are reasonable.
\begin{example}[Uniform positive definiteness condition does not hold, \cite{Waurick2012aDIE}]
 Let $H=\mathbb{C}$, $\nu>0$ and, for $n\in\mathbb{N}$, let $\mathcal{M}_{n}=\partial_{0}^{-1}\frac{1}{n}$,
$f\in L_{\nu}^{2}(\mathbb{R})\setminus\{0\}$. For $n\in\mathbb{N}$,
let $u_{n}\in L_{\nu}^{2}(\mathbb{R})$ be defined by 
\[
\partial_{0}\mathcal{M}_{n}u_{n}=\frac{1}{n}u_{n}=f.
\]
 Then $(u_{n})_{n}$ is not relatively weakly compact and contains
no weakly convergent subsequence. 
\end{example}
~
\begin{example}[Boundedness assumption does not hold]
 Let $H=\mathbb{C}$, $\nu>0$ and, for $n\in\mathbb{N}$, let $\mathcal{M}_{n}=\partial_{0}^{-1}n$,
$f\in L_{\nu}^{2}(\mathbb{R})$. For $n\in\mathbb{N}$, let $u_{n}\in L_{\nu}^{2}(\mathbb{R})$
be defined by 
\[
\partial_{0}\mathcal{M}_{n}u_{n}=nu_{n}=f.
\]
 Then $(u_{n})_{n}$ converges to $0$. Thus, a limit ``equation''
would be in fact the relation $\{0\}\times L_{\nu}^{2}(\mathbb{R})\subseteqq L_{\nu}^{2}(\mathbb{R})\oplus L_{\nu}^{2}(\mathbb{R})$.
\end{example}
We will now apply our main theorems to particular situations.

\section{Time-independent coefficients\label{sec:Time-independent-coefficients}}

In this section, we treat time-independent coefficients. That is to
say, we assume that the operators in the sequences under consideration
only act on the ``spatial'' Hilbert spaces $H_{0}$ and $H_{1}$
in Theorem \ref{thm:second-hom_thm} or $H$ in Theorem \ref{thm:first-hom_thm}.
More precisely and similar to Example \ref{Ex:time-in}, for a bounded
linear operator $M\in L(H_{0},H_{1})$ there is a (canonical) extension
to $L_{\nu}^{2}$-functions in the way that $(M\phi)(t)\coloneqq M(\phi(t))$
for $\phi\in L_{\nu}^{2}(\mathbb{R};H_{0})$ and a.e.\  $t\in\mathbb{R}$.
Thus $M$ is evolutionary (Example \ref{Ex:time-in}). We only state
the specialization of this situation for Theorem \ref{thm:first-hom_thm}.
The result reads as follows.
\begin{thm}
\label{thm:time-independent} Let $H$ be a separable Hilbert space,
$\nu_{0}>0$. Let $\left(M_{n}\right)_{n}$, $\left(N_{n}\right)_{n}$
be bounded sequences in $L(H)$. Assume there exists $c>0$ such that
for all $n\in\mathbb{N}$ 
\[
\Re\langle M_{n}\phi,\phi\rangle_{H}\geq c\langle\phi,\phi\rangle_{H}\quad(\phi\in H).
\]
Then there exists $\nu\geqq\nu_{0}$ and a subsequence $(n_{k})_{k}$
of $(n)_{n}$ such that
\[
\partial_{0}M_{n_{k}}+N_{n_{k}}\stackrel{G}{\longrightarrow}\partial_{0}M_{hom,0}^{-1}+\partial_{0}\sum_{j=1}^{\infty}\left(-\sum_{\ell=1}^{\infty}M_{hom,0}^{-1}M_{hom,\ell}\left(-\partial_{0}^{-1}\right)^{\ell}\right)^{j}M_{hom,0}^{-1},
\]
as $k\to\infty$ in $L_{\nu}^{2}(\mathbb{R};H)$, where 
\[
M_{hom,0}=\tau_{\textnormal{w}}\textnormal{-}\lim_{k\to\infty}M_{n_{k}}^{-1}
\]
 and 
\[
M_{hom,\ell}=\tau_{\textnormal{w}}\textnormal{-}\lim_{k\to\infty}M_{n_{k}}^{-1}\left(N_{n_{K}}M_{n_{K}}^{-1}\right)^{\ell}.
\]
\end{thm}
\begin{proof}
At first observe that $M\partial_{0}^{-1}=\partial_{0}^{-1}M$ for
all bounded linear operators $M\in L(H)$. Moreover, the estimate
$\Re\langle M\phi,\phi\rangle_{H}\geqq c\langle\phi,\phi\rangle_{H}$
for $\phi\in H$ also carries over to the analogous one for $\phi\in L_{\nu}^{2}(\mathbb{R};H)$
and the extended $M$. Hence, the result follows from Theorem \ref{thm:first-hom_thm}.\end{proof}
\begin{rem}
As it has already been observed in \cite{Mascarenhas1984,TartarNonlHom},
the class of equations treated in Theorem \ref{thm:time-independent}
is \emph{not }closed under $G$-convergence in general. The next example
shows that this effect only occurs if the Hilbert space $H$ is infinite-dimensional
and the convergence of $\left(M_{n}\right)_{n}$ and $\left(N_{n}\right)_{n}$
is ``weak enough'' in a sense to be specified below.\end{rem}
\begin{example}
Assume that $H$ is finite-dimensional. Then $(M_{n})_{n}$ and $(N_{n})_{n}$
are a mere bounded sequences of matrices with constant coefficients.
In particular, the weak operator topology coincides with the topology
induced by the operator norm. Hence, the processes of computing the
inverse and computing the limit interchange and multiplication is
a continuous process as well. Thus, assuming $(M_{n})_{n}$ and $(N_{n})_{n}$
to be convergent with the respective limits $M$ and $N$, we compute 

{\footnotesize
\begin{align*}
 & \partial_{0}M_{hom,0}^{-1}+\partial_{0}\sum_{j=1}^{\infty}\left(-\sum_{\ell=1}^{\infty}M_{hom,0}^{-1}M_{hom,\ell}\left(-\partial_{0}^{-1}\right)^{\ell}\right)^{j}M_{hom,0}^{-1}\\
 & =\partial_{0}\left(\tau_{\textnormal{w}}\textnormal{-}\lim_{k\to\infty}M_{n_{k}}^{-1}\right)^{-1}\\
 & \quad+\partial_{0}\sum_{j=1}^{\infty}\left(-\sum_{\ell=1}^{\infty}\left(\tau_{\textnormal{w}}\textnormal{-}\lim_{k\to\infty}M_{n_{k}}^{-1}\right)^{-1}\left(\tau_{\textnormal{w}}\textnormal{-}\lim_{k\to\infty}M_{n_{k}}^{-1}\left(N_{n_{k}}M_{n_{k}}^{-1}\right)^{\ell}\right)\left(-\partial_{0}^{-1}\right)^{\ell}\right)^{j}\left(\tau_{\textnormal{w}}\textnormal{-}\lim_{k\to\infty}M_{n_{k}}^{-1}\right)^{-1}\\
 & =\partial_{0}\left(\lim_{k\to\infty}M_{n_{k}}^{-1}\right)^{-1}\\
 & \quad+\partial_{0}\sum_{j=1}^{\infty}\left(-\sum_{\ell=1}^{\infty}\left(\lim_{k\to\infty}M_{n_{k}}^{-1}\right)^{-1}\left(\lim_{k\to\infty}M_{n_{k}}^{-1}\left(N_{n_{k}}M_{n_{k}}^{-1}\right)^{\ell}\right)\left(-\partial_{0}^{-1}\right)^{\ell}\right)^{j}\left(\lim_{k\to\infty}M_{n_{k}}^{-1}\right)^{-1}\\
 & =\partial_{0}M+\partial_{0}\sum_{j=1}^{\infty}\left(-\sum_{\ell=1}^{\infty}M\left(M^{-1}\left(NM^{-1}\right)^{\ell}\right)\left(-\partial_{0}^{-1}\right)^{\ell}\right)^{j}M\\
 & =\partial_{0}M+\partial_{0}\sum_{j=1}^{\infty}\left(-\sum_{\ell=1}^{\infty}\left(NM^{-1}\right)^{\ell}\left(-\partial_{0}^{-1}\right)^{\ell}\right)^{j}M=\partial_{0}\sum_{j=0}^{\infty}\left(-\sum_{\ell=1}^{\infty}\left(NM^{-1}\right)^{\ell}\left(-\partial_{0}^{-1}\right)^{\ell}\right)^{j}M\\
 & =\partial_{0}\left(1+\sum_{\ell=1}^{\infty}\left(NM^{-1}\right)^{\ell}\left(-\partial_{0}^{-1}\right)^{\ell}\right)^{-1}M=\partial_{0}\left(\sum_{\ell=0}^{\infty}\left(-NM^{-1}\partial_{0}^{-1}\right)^{\ell}\right)^{-1}M\\
 & =\partial_{0}\left(\left(1+NM^{-1}\partial_{0}^{-1}\right)^{-1}\right)^{-1}M=\partial_{0}\left(M+N\partial_{0}^{-1}\right)=\partial_{0}M+N.
\end{align*}
}

Thus, in finite-dimensional spaces, the above theorem restates the
continuous dependence of the solution on the coefficients. Note that
we only used that multiplication and computing the inverse are continuous
operations. Hence, the above calculation literally expresses the fact
of continuous dependence on the coefficients if $H$ is infinite-dimensional
and the sequences $\left(M_{n}\right)_{n}$ and $\left(N_{n}\right)_{n}$
converge in the strong operator topology. Thus, one can only expect
that the limit expression differs from the one, which one might expect,
if the actual convergence of the operators involved is strictly weaker
than in the strong operator topology.
\end{example}
We will turn to a more sophisticated example. For this we recall the
concept of periodicity in $\mathbb{R}^{n}$, see e.g.\ \cite{CioDon}.
\begin{defn}
Let $a\colon\mathbb{R}^{n}\to\mathbb{C}^{m\times m}$ be bounded and
measurable. $a$ is called \emph{$]0,1[^{n}$-periodic}, if for all
$x\in\mathbb{R}^{n}$ and $k\in\mathbb{Z}^{n}$ we have $a(x+k)=a(x)$.\emph{ }
\end{defn}
Moreover, recall the following well-known convergence result on periodic
mappings, cf.\ e.g.\ \cite[Theorem 2.6]{CioDon}. 
\begin{thm}
\label{thm:per}Let $a\colon\mathbb{R}^{n}\to\mathbb{C}^{m\times m}$
be bounded and measurable\emph{ and $]0,1[^{n}$-periodic. Then $\left(a(k\cdot)\right)_{k}$
converges in $L^{\infty}(\mathbb{R}^{n})^{m\times m}$ $*$-weakly
to the integral mean $\int_{[0,1]^{n}}a(y)dy$.}\end{thm}
\begin{rem}
For any bounded measurable function $a\colon\mathbb{R}^{n}\to\mathbb{C}^{m\times m}$
one can associate the corresponding multiplication operator in $L^{2}(\mathbb{R}^{n})^{m}$.
Hence, Theorem \ref{thm:per} states the fact that in case of periodic
$a$ the sequence of associated multiplication operators of $a(k\cdot)$
converges in the weak operator topology to the operator of multiplying
with the respective integral mean. Indeed, this follows easily from
$L^{2}(\mathbb{R}^{n})\cdot L^{2}(\mathbb{R}^{n})=L^{1}(\mathbb{R}^{n})$.
See also \cite[8.10]{Conway}.\end{rem}
\begin{example}
\label{Ex:ODE_indep} Let $H=L^{2}\left(\mathbb{R}^{n}\right)^{m}$
and let $a,b\colon\mathbb{R}^{n}\to\mathbb{C}^{m\times m}$ be bounded,
measurable and $]0,1[^{n}$-periodic. We assume $\Re a(x)\geqq c$
for all $x\in\mathbb{R}^{n}$. Observe that any polynomial in $a$
and $b$ is $]0,1[^{n}$-periodic and so is $a^{-1}\coloneqq\left(x\mapsto a(x)^{-1}\right).$
Thus, by Theorem \ref{thm:time-independent}, we deduce that {\footnotesize
\begin{align*}
 & \partial_{0}a(k\cdot)+b(k\cdot)\\
 & \stackrel{G}{\longrightarrow}\partial_{0}\left(\int_{[0,1]^{n}}a(y)^{-1}dy\right)^{-1}\\
 & \quad\quad+\partial_{0}\sum_{j=1}^{\infty}\left(-\sum_{\ell=1}^{\infty}\left(\int_{[0,1]^{n}}a(y)^{-1}dy\right)^{-1}\int_{[0,1]^{n}}a(y)^{-1}\left(b(y)a(y)^{-1}\right)^{\ell}dy\left(-\partial_{0}^{-1}\right)^{\ell}\right)^{j}\left(\int_{[0,1]^{n}}a(y)^{-1}dy\right)^{-1},
\end{align*}
} as $k\to\infty$ in $L_{\nu}^{2}(\mathbb{R};H)$. \end{example}
\begin{rem}
In \cite[Theorem 1.2]{Tartar_Mem} or \cite{TartarNonlHom} the author
considers the equation $\left(\partial_{0}+b_{k}(\cdot)\right)u_{k}=f$
with $(b_{k})_{k}$ being a $[\alpha,\beta]$-valued (for some $\alpha,\beta\in\mathbb{R})$
sequence of bounded, measurable mappings depending on one spatial
variable. Also in that exposition a memory effect is derived. However,
the method uses the concept of Young measures. The reason for that
is the representation of the solution being a function of the oscillating
coefficent. More precisely, the convergence of the sequence $\left(e^{tb(k\cdot)}\right)_{k}$
is addressed. In order to let $k$ tend to infinity in this expression
one needs a result on the (weak-$*$) convergence of (continuous)
functions of bounded functions. This is where the Young-measures come
into play, see e.g.\ \cite[Section 2]{Ball} and the references therein
or \cite[p. 930]{TartarNonlHom}. The result used is the following.
There exists a family of probabilty measures $\left(\nu_{x}\right)_{x}$
supported on $[\alpha,\beta]$ such that for (a subsequence of) $(k)_{k}$
and all real continuous functions $G$ we have
\[
G\circ b_{k}(\cdot)\to\left(\mathbb{R}\ni x\mapsto\int_{[\alpha,\beta]}G(\lambda)d\nu_{x}(\lambda)\right)
\]
as $k\to\infty$ in $L^{\infty}\left(\mathbb{R}\right)$ $*$-weakly.
The family $\left(\nu_{x}\right)_{x}$ is also called the \emph{Young-measure
associated to $\left(b_{k}\right)_{k}$.} With the help of the family
$\left(\nu_{x}\right)_{x}$ a convolution kernel is computed such
that the respective limit equation can be written as 
\[
\partial_{0}u(t,x)+b^{0}(x)u-\int_{0}^{t}K(x,t-s)u(x,s)ds=f(x,t),
\]
where $b^{0}$ is a weak-$*$-limit of a subsequence of $(b_{k})_{k}$
and $K(x,t)=\int_{\mathbb{R}_{>0}}e^{-\lambda t}d\nu_{x}(\lambda)$
for a.e. $t\in\mathbb{R}_{>0}$ and $x\in\mathbb{R}$. The relation
to our above considerations is as follows. The resulting limit equation
within our approach can also be considered as an ordinary differential
equation perturbed by a convolution term. Denoting limits with respect
to the $\sigma(L_{\infty},L_{1})$-topology by $*$-$\lim$, we realize
that Theorem \ref{thm:time-independent} in this particular situation
states that the limit equation admits the form
\begin{align*}
 & \partial_{0}+\partial_{0}\sum_{k=1}^{\infty}\left(-\sum_{\ell=1}^{\infty}*\text{-}\lim_{k\to\infty}\left(b_{k}\right)^{\ell}\left(-\partial_{0}^{-1}\right)^{\ell}\right)^{k}\\
 & =\partial_{0}+b^{0}+\sum_{\ell=2}^{\infty}*\text{-}\lim_{k\to\infty}\left(b_{k}\right)^{\ell}\left(-\partial_{0}^{-1}\right)^{\ell-1}+\partial_{0}\sum_{k=2}^{\infty}\left(-\sum_{\ell=1}^{\infty}*\text{-}\lim_{k\to\infty}\left(b_{k}\right)^{\ell}\left(-\partial_{0}^{-1}\right)^{\ell}\right)^{k}
\end{align*}
as $k\to\infty$ in $L_{\nu}^{2}(\mathbb{R};L^{2}(\mathbb{R}))$.
Indeed, using \cite[6.2.6. Memory Problems, (b) p. 448]{Picard} or
\cite[Theorem 1.5.6 and Remark 1.5.7]{Waurick2011Diss}, we deduce
that the term 
\[
\sum_{\ell=2}^{\infty}*\text{-}\lim_{k\to\infty}\left(b_{k}\right)^{\ell}\left(-\partial_{0}^{-1}\right)^{\ell-1}+\partial_{0}\sum_{k=2}^{\infty}\left(-\sum_{\ell=1}^{\infty}*\text{-}\lim_{k\to\infty}\left(b_{k}\right)^{\ell}\left(-\partial_{0}^{-1}\right)^{\ell}\right)^{k}
\]
can be represented as a (temporal) convolution. Moreover, note that
the choice of subsequences is the same. Indeed, in the above rationale
with the Young measure approach, by a density argument, it suffices
to choose a subsequence of $(b_{k})_{k}$ such that any polynomial
of $(b_{k})_{k}$ converges $*$-weakly. This choice of subsequences
also suffices to deduce $G$-convergence of the respective equations
within the operator-theoretic perspective treated in this exposition.
\end{rem}
In the next example, we consider a partial differential equation,
which can be reformulated as ordinary differential equation in an
infinite-dimensional Hilbert space. More precisely, we treat Maxwell's
equations with the Drude-Born-Fedorov material model, see e.g. \cite{Picard-Freymond}.
In order to discuss this equation properly, we need to introduce several
operators from vector analysis.
\begin{defn}
Let $\Omega\subseteqq\mathbb{R}^{3}$ be open. Then we define%
\footnote{We denote by $C_{\infty,c}(\Omega)$ the set of arbitrarily often
differentiable functions with compact support in $\Omega$. %
} 
\begin{align*}
\curl_{c}\colon C_{\infty,c}(\Omega)^{3}\subseteqq L^{2}(\Omega)^{3} & \to L^{2}(\Omega)^{3}\\
\phi & \mapsto\left(\begin{array}{ccc}
0 & -\partial_{3} & \partial_{2}\\
\partial_{3} & 0 & -\partial_{1}\\
-\partial_{2} & \partial_{1} & 0
\end{array}\right)\phi,
\end{align*}
where we denote by $\partial_{i}$ the partial derivative with respect
to the $i$'th variable, $i\in\{1,2,3\}$. Moreover, introduce 
\begin{align*}
\dive_{c}\colon C_{\infty,c}(\Omega)^{3}\subseteqq L^{2}(\Omega)^{3} & \to L^{2}(\Omega)\\
\left(\phi_{1},\phi_{2},\phi_{3}\right) & \mapsto\sum_{i=1}^{3}\partial_{i}\phi_{i}.
\end{align*}
We define $\curl_{0}\coloneqq\overline{\curl_{c}}$, $\dive_{0}\coloneqq\overline{\dive_{c}}$.
The $0$ serves as a reminder for (the generalization of) the electric
and the Neumann boundary condition, respectively. If $\Omega$ is
simply connected, we also introduce 
\begin{align*}
\curl_{\diamond}\colon D(\curl_{\diamond})\subseteqq L^{2}(\Omega)^{3} & \to L^{2}(\Omega)^{3}\\
\phi & \mapsto\curl\phi,
\end{align*}
where $D(\curl_{\diamond})\coloneqq\left\{ \phi\in D(\curl);\curl\phi\in D(\dive_{0})\right\} .$\end{defn}
\begin{rem}
\label{curl-selfadjoint}It can be shown that if $\Omega$ is simply
connected with finite measure, then $\curl_{\diamond}$ is a selfadjoint
operator, see \cite{Filonov,Picard-Freymond}. In that reference it
is also stated that $\curl_{\diamond}$ has, except $0$, only discrete
spectrum. In particular, this means that the intersection of the resolvent
set of $\curl_{\diamond}$ with $\mathbb{R}$ is non-empty. For other
geometric properties of $\Omega$ resulting in the selfadjointness
of $\curl_{\diamond}$, we refer to \cite{Picard_curl_ext}.
\end{rem}
We now treat a homogenization problem of the Drude-Born-Fedorov model
as treated in \cite{Picard-Freymond}.
\begin{example}
\label{Drude-Born-Fedorov-Ex} Assume that $\Omega\subseteqq\mathbb{R}^{3}$
is open, simply connected and has bounded Lebesgue measure. Invoking
Remark \ref{curl-selfadjoint} and following \cite[Theorem 2.1]{Picard-Freymond},
the equation
\begin{equation}
\left(\partial_{0}\left(1+\eta\curl_{\diamond}\right)\left(\begin{array}{cc}
\epsilon & 0\\
0 & \mu
\end{array}\right)+\left(\begin{array}{cc}
0 & -\curl_{\diamond}\\
\curl_{\diamond} & 0
\end{array}\right)\right)\left(\begin{array}{c}
E\\
H
\end{array}\right)=\left(\begin{array}{c}
J\\
0
\end{array}\right)\label{eq:DBF}
\end{equation}
for $\eta\in\mathbb{R}$ such that $-\frac{1}{\eta}\in\rho(\curl_{\diamond})$,
$J\in L_{\nu}^{2}(\mathbb{R};L^{2}(\Omega)^{3})$ and given $\epsilon,\mu\in L(L^{2}(\Omega)^{3})$
being strictly positive selfadjoint operators, admits a unique solution
$(E,H)\in H_{\nu,1}(\mathbb{R};L^{2}(\Omega)^{3}).$%
\footnote{Note that for $(E,H)\in H_{\nu,1}(\mathbb{R};L^{2}(\Omega)^{3})$
being a solution of (\ref{eq:DBF}) can only be true in the distributional
sense, which can be made more precise with the help of the extrapolation
spaces of $\curl_{\diamond}$. We shall, however, not follow this
reasoning here in more details and refer again to \cite{Picard-Freymond}
or \cite[Chapter 2]{Picard}.%
} Indeed, multiplying (\ref{eq:DBF}) by $\left(1+\eta\curl_{\diamond}\right)^{-1},$
we get that 
\[
\left(\partial_{0}\left(\begin{array}{cc}
\epsilon & 0\\
0 & \mu
\end{array}\right)+\curl_{\diamond}\left(1+\eta\curl_{\diamond}\right)^{-1}\left(\begin{array}{cc}
0 & -1\\
1 & 0
\end{array}\right)\right)\left(\begin{array}{c}
E\\
H
\end{array}\right)=\left(1+\eta\curl_{\diamond}\right)^{-1}\left(\begin{array}{c}
J\\
0
\end{array}\right).
\]
Realizing that \foreignlanguage{english}{$\curl_{\diamond}\left(1+\eta\curl_{\diamond}\right)^{-1}$}
is a \emph{bounded }linear operator by the spectral theorem for the
selfadjoint operator $\curl_{\diamond}$, we get that $(E,H)\in H_{\nu,1}(\mathbb{R};L^{2}(\Omega)^{3})$
solves the above equation. Note that the equation derived from (\ref{eq:DBF})
is a mere ordinary differential equation in an infinite-dimensional
Hilbert space. Assume we are given bounded sequences of selfadjoint
operators $\left(\epsilon_{n}\right)_{n}$ and $\left(\mu_{n}\right)_{n}$
satisfying $\epsilon_{n}\geqq c$ and $\mu_{n}\geqq c$ for some $c>0$
and all $n\in\mathbb{N}$. For $n\in\mathbb{N}$ we consider the problem
\[
\left(\partial_{0}\left(\begin{array}{cc}
\epsilon_{n} & 0\\
0 & \mu_{n}
\end{array}\right)+\curl_{\diamond}\left(1+\eta\curl_{\diamond}\right)^{-1}\left(\begin{array}{cc}
0 & -1\\
1 & 0
\end{array}\right)\right)\left(\begin{array}{c}
E_{n}\\
H_{n}
\end{array}\right)=\left(1+\eta\curl_{\diamond}\right)^{-1}\left(\begin{array}{c}
J\\
0
\end{array}\right)
\]
and address the question of $G$-convergence of (a subsequence of)
\[
\left(\text{DBF}_{n}\right)_{n}\coloneqq\left(\partial_{0}\left(\begin{array}{cc}
\epsilon_{n} & 0\\
0 & \mu_{n}
\end{array}\right)+\curl_{\diamond}\left(1+\eta\curl_{\diamond}\right)^{-1}\left(\begin{array}{cc}
0 & -1\\
1 & 0
\end{array}\right)\right)_{n}.
\]
Clearly, Theorem \ref{thm:time-independent} applies and we get that
(a subsequence of) $\left(\text{DBF}_{n}\right)_{n}$ $G$-converges
to 
\[
\partial_{0}M_{hom,0}^{-1}+\partial_{0}\sum_{k=1}^{\infty}\left(-\sum_{\ell=1}^{\infty}M_{hom,0}^{-1}M_{hom,\ell}\left(-\partial_{0}^{-1}\right)^{\ell}\right)^{k}M_{hom,0}^{-1},
\]
as $k\to\infty$ in $L_{\nu}^{2}(\mathbb{R};H)$, where 
\[
M_{hom,0}=\tau_{\textnormal{w}}\textnormal{-}\lim_{k\to\infty}\left(\begin{array}{cc}
\epsilon_{n_{k}}^{-1} & 0\\
0 & \mu_{n_{k}}^{-1}
\end{array}\right)
\]
 and 
\[
M_{hom,\ell}=\tau_{\textnormal{w}}\textnormal{-}\lim_{k\to\infty}\left(\begin{array}{cc}
\epsilon_{n_{k}}^{-1} & 0\\
0 & \mu_{n_{k}}^{-1}
\end{array}\right)\left(\curl_{\diamond}\left(1+\eta\curl_{\diamond}\right)^{-1}\left(\begin{array}{cc}
0 & -\mu_{n_{k}}^{-1}\\
\epsilon_{n_{k}}^{-1} & 0
\end{array}\right)\right)^{\ell}.
\]

\end{example}
We have seen that the class of problems discussed in Theorem \ref{thm:time-independent}
in this section is \emph{not }closed under the $G$-convergence, unless
$N=0$.

\section{Time-translation invariant coeffcients\label{sec:Time-translation-invariant-coeff}}

In Theorem \ref{thm:time-independent}, we have seen that the limit
equation can be described as a power series expression in $\partial_{0}^{-1}$.
A possible way to generalize this is the introduction of holomorphic
functions in $\partial_{0}^{-1}$, see \cite[Section 6.1, page 427]{Picard}.
To make this precise, we need the spectral representation for $\partial_{0}^{-1}$,
the Fourier-Laplace transform $\mathcal{L}_{\nu},$ which is given
as the unitary operator being the closure of 
\begin{align*}
C_{\infty,c}(\mathbb{R})\subseteqq L_{\nu}^{2}(\mathbb{R}) & \to L_{\nu}^{2}(\mathbb{R})\\
\phi & \mapsto\left(x\mapsto\frac{1}{\sqrt{2\pi}}\int_{\mathbb{R}}e^{-ixy-\nu y}\phi(y)dy\right).
\end{align*}
Denoting by $m\colon D(m)\subseteqq L^{2}(\mathbb{R})\to L^{2}(\mathbb{R}),f\mapsto\left(x\mapsto xf(x)\right)$
the multiplication-by-argument-operator with maximal domain $D(m)$,
we arrive at the representation
\[
\partial_{0}^{-1}=\mathcal{L}_{\nu}^{*}\left(\frac{1}{im+\nu}\right)\mathcal{L}_{\nu}.
\]
Thus, for bounded and analytic functions $M\colon B(r,r)\to\mathbb{C}$
with $r>\frac{1}{2\nu}$ we define
\[
M\left(\partial_{0}^{-1}\right)\coloneqq\mathcal{L}_{\nu}^{*}M\left(\frac{1}{im+\nu}\right)\mathcal{L}_{\nu},
\]
where $\left(M\left(\frac{1}{im+\nu}\right)\phi\right)(t)\coloneqq M\left(\frac{1}{it+\nu}\right)\phi(t)$
for $\phi\in L^{2}(\mathbb{R})$ and a.e.\  $t\in\mathbb{R}$. We
canonically extend the above definitions to the case of vector-valued
functions $L_{\nu}^{2}(\mathbb{R};H)$ with values in a Hilbert space
$H$ . In this way, the definition of $M\left(\partial_{0}^{-1}\right)$
can be generalized to bounded and operator-valued functions $M:B(r,r)\to L(H_{0},H_{1})$
for Hilbert spaces $H_{0}$ and $H_{1}$. We denote
\[
\mathcal{H}^{\infty}(B(r,r);L(H_{0},H_{1}))\coloneqq\left\{ M:B(r,r)\to L(H_{0},H_{1});M\text{ bounded, analytic}\right\} .
\]
A subset $\mathfrak{M}\subseteqq\mathcal{H}^{\infty}(B(r,r);L(H_{0},H_{1}))$
is called \emph{bounded}, if 
\[
\sup\{\|M(z)\|;z\in B(r,r),M\in\mathfrak{M}\}<\infty.
\]
A family $(M_{\iota})_{\iota\in I}$ in $\mathcal{H}^{\infty}(B(r,r);L(H_{0},H_{1}))$
is \emph{bounded},\emph{ }if $\left\{ M_{\iota};\iota\in I\right\} $
is bounded. 

We will treat some examples for $\mathcal{H}^{\infty}$-functions
of $\partial_{0}^{-1}$ below, see also \cite{WaurickMMA}. In this
reference, a homogenization theorem of problems of the kind treated
in Theorem \ref{thm:first-hom_thm} with $\left(\mathcal{M}_{n}\right)_{n}=\left(M_{n}\left(\partial_{0}^{-1}\right)\right)_{n}$
for a bounded sequence $\left(M_{n}\right)_{n}$ in $\mathcal{H}^{\infty}$
has been presented, see \cite[Theorem 5.2]{WaurickMMA}. Moreover,
in \cite[Theorem 4.4]{Waurick2012aDIE} a special case of an analogous
result of Theorem \ref{thm:second-hom_thm} has been presented and
used. In order to state a $G$-convergence theorem in a more general
situation, note that 
\[
\left\{ M\left(\partial_{0}^{-1}\right);M\in\mathcal{H}^{\infty}(B(r,r);L(H_{0},H_{1}))\right\} \subseteqq\bigcap_{\frac{1}{2r}<\nu}L_{\text{ev},\nu}(H_{0},H_{1}).
\]
The theorem reads as follows.
\begin{thm}
\label{thm:translation-inv} Let $H_{0},H_{1}$ be separable Hilbert
spaces, $\nu_{0}>0$, $r>\frac{1}{2\nu_{0}}$. Let $\left(M_{n}\right)_{n},\left(N_{n}^{ij}\right)_{n}$
be bounded sequences in $\mathcal{H}^{\infty}(B(r,r);L(H_{0}))$ and
$\mathcal{H}^{\infty}(B(r,r);L(H_{j},H_{i}))$, respectively ($i,j\in\{0,1\}$).
Assume there exists $c>0$ such that for all $n\in\mathbb{N}$ we
have for all $(\phi,\psi)\in H_{0}\oplus H_{1}$ and $z\in B(r,r)$
\[
\Re\langle M_{n}(z)\phi,\phi\rangle_{H_{0}}\geqq c\left|\phi\right|_{H_{0}}^{2},\quad\Re\langle N_{n}^{11}(z)\psi,\psi\rangle_{H_{1}}\geqq c\left|\psi\right|_{H_{1}}^{2}.
\]
Then there exists $\nu>\nu_{0}$ and a subsequence $(n_{k})_{k}$
of $(n)_{n}$ such that {\footnotesize 
\begin{align*}
 & \partial_{0}\left(\begin{array}{cc}
M_{n_{k}}\left(\partial_{0}^{-1}\right) & 0\\
0 & 0
\end{array}\right)+\left(\begin{array}{cc}
N_{n_{k}}^{00}\left(\partial_{0}^{-1}\right) & N_{n_{k}}^{01}\left(\partial_{0}^{-1}\right)\\
N_{n_{k}}^{10}\left(\partial_{0}^{-1}\right) & N_{n_{k}}^{11}\left(\partial_{0}^{-1}\right)
\end{array}\right)\\
 & \stackrel{G}{\longrightarrow}\left(\begin{array}{cc}
\partial_{0} & 0\\
0 & 1
\end{array}\right)\left(\left(\begin{array}{cc}
M_{hom,0,00}\left(\partial_{0}^{-1}\right)^{-1} & 0\\
0 & N_{hom,-1,11}\left(\partial_{0}^{-1}\right)^{-1}
\end{array}\right)^{\phantom{\ell}}\right.\\
 & \left.+\sum_{\ell=1}^{\infty}\left(-\left(\begin{array}{cc}
M_{hom,0,00}\left(\partial_{0}^{-1}\right)^{-1} & 0\\
0 & N_{hom,-1,11}\left(\partial_{0}^{-1}\right)^{-1}
\end{array}\right)M^{(1)}\left(\partial_{0}^{-1}\right)\right)^{\ell}\left(\begin{array}{cc}
M_{hom,0,00}\left(\partial_{0}^{-1}\right)^{-1} & 0\\
0 & N_{hom,-1,11}\left(\partial_{0}^{-1}\right)^{-1}
\end{array}\right)\right),
\end{align*}
}where we put 
\[
N_{n}\coloneqq N_{n}^{00}-N_{n}^{01}\left(N_{n}^{11}\right)^{-1}N_{n}^{10}\quad(n\in\mathbb{N})
\]
as well as 
\[
M^{(1)}\left(\partial_{0}^{-1}\right)\coloneqq\left(\begin{array}{cc}
\sum_{\ell=1}^{\infty}M_{hom,\ell,00}\left(\partial_{0}^{-1}\right)\left(\partial_{0}^{-1}\right)^{\ell} & \sum_{\ell=0}^{\infty}M_{hom,\ell,01}\left(\partial_{0}^{-1}\right)\left(\partial_{0}^{-1}\right)^{\ell+1}\\
\sum_{\ell=0}^{\infty}M_{hom,\ell,10}\left(\partial_{0}^{-1}\right)\left(\partial_{0}^{-1}\right)^{\ell} & \sum_{\ell=0}^{\infty}M_{hom,\ell,11}\left(\partial_{0}^{-1}\right)\left(\partial_{0}^{-1}\right)^{\ell+1}
\end{array}\right)
\]
and 
\begin{align*}
M_{hom,\ell,00}(z) & =\tau_{\textnormal{w}}\text{-}\lim_{k\to\infty}M_{n_{k}}(z)^{-1}\left(-N_{n_{k}}(z)M_{n_{k}}(z)^{-1}\right)^{\ell},\\
M_{hom,\ell,01}(z) & =\tau_{\textnormal{w}}\text{-}\lim_{k\to\infty}-M_{n_{k}}(z)^{-1}\left(-N_{n_{k}}(z)M_{n_{k}}(z)^{-1}\right)^{\ell}N_{n_{k}}^{01}(z)\left(N_{n_{k}}^{11}(z)\right)^{-1},\\
M_{hom,\ell,10}(z) & =\tau_{\textnormal{w}}\text{-}\lim_{k\to\infty}-\left(N_{n_{k}}^{11}(z)\right)^{-1}N_{n_{k}}^{10}(z)M_{n_{k}}(z)^{-1}\left(-N_{n_{k}}(z)M_{n_{k}}(z)^{-1}\right)^{\ell},\\
M_{hom,\ell,11}(z) & =\tau_{\textnormal{w}}\text{-}\lim_{k\to\infty}\left(N_{n_{k}}^{11}(z)\right)^{-1}N_{n_{k}}^{10}(z)M_{n_{k}}(z)^{-1}\left(-N_{n_{k}}(z)M_{n_{k}}(z)^{-1}\right)^{\ell}N_{n_{k}}^{01}(z)\left(N_{n_{k}}^{11}(z)\right)^{-1},\\
N_{hom,-1,11}(z) & =\tau_{\textnormal{w}}\text{-}\lim_{k\to\infty}\left(N_{n_{k}}^{11}(z)\right)^{-1},
\end{align*}
for all $z\in B\left(\frac{1}{2\nu_{1}},\frac{1}{2\nu_{1}}\right)$
for some $\nu>\nu_{1}\geqq\nu_{0}.$\end{thm}
\begin{proof}
Observe that bounded and analytic functions of $\partial_{0}^{-1}$
commute with $\partial_{0}^{-1}$. Note that the only thing left to
prove is that the operator-valued functions involved are indeed analytic
functions of $\partial_{0}^{-1}$. For this we need to introduce a
topology on $\mathcal{H}^{\infty}(B(r,r);L(H_{0},H_{1})).$ Let $\tau$
be the topology induced by the mappings
\begin{align*}
\mathcal{H}^{\infty}(B(r,r);L(H_{0},H_{1})) & \to\mathcal{H}(B(r,r))\\
M & \mapsto\langle\phi,M(\cdot)\psi\rangle,
\end{align*}
for all $(\phi,\psi)\in H_{1}\oplus H_{0},$ where $\mathcal{H}(B(r,r))$
is the set of analytic functions endowed with the compact open topology.
In \cite[Theorem 3.4]{WaurickMMA} or \cite[Theorem 4.3]{Waurick2013frac}
it is shown that closed and bounded subsets of $\mathcal{H}^{\infty}(B(r,r);L(H_{0},H_{1}))$
are sequentially compact with respect to the $\tau$-topology. %
{} Furthermore, by \cite[Lemma 3.5]{WaurickMMA}, we have that if a
bounded sequence $\left(T_{n}\right)_{n}$ in $\mathcal{H}^{\infty}(B(r,r);L(H_{0},H_{1}))$
converges in the $\tau$-topology then the operator sequence $\left(T_{n}\left(\partial_{0}^{-1}\right)\right)_{n}$
converges in the weak operator topology of $L\left(L_{\nu}^{2}(\mathbb{R};H_{0}\oplus H_{1})\right)$.
Putting all this together, we deduce that the assertion follows from
Theorem \ref{thm:second-hom_thm}.\end{proof}
\begin{rem}
(a) Theorem \ref{thm:translation-inv} asserts that the time-translation
invariant equations under consideration are closed under $G$-convergence.
Though the formulas may become a bit cluttered, in principle, an iterated
homogenization procedure is possible.

(b) In \cite[Theorem 4.4]{Waurick2012aDIE} operator-valued functions
that are analytic at $0$ were treated. This assumption can be lifted.
Indeed, we only require analyticity of the operator-valued functions
under consideration on the open ball $B(r,r)$ for some radius $r>0$
and do not assume that any of these functions have holomorphic extensions
to $0$.
\end{rem}
We give several examples. 
\begin{example}
\label{Ex:conv}Let $\nu_{0}>0$. In this example we treat integral
equations of convolution type. Let $\left(g_{n}\right)_{n}$ be a
bounded sequence in $L_{\nu_{0}}^{1}(\mathbb{R}_{>0})$ such that
there is $h\in L_{\nu_{0}}^{1}(\mathbb{R}_{>0})$ with $\|g(t)\|\leqq h(t)$
for all $n\in\mathbb{N}$ and a.e. $t\in\mathbb{R}.$ For $f\in C_{\infty,c}(\mathbb{R})$
consider the equation
\begin{equation}
u_{n}+g_{n}*u_{n}=f.\label{eq:Conv-eq}
\end{equation}
The latter equation fits into the scheme of Theorem \ref{thm:translation-inv}
for $H=\mathbb{C}$. Indeed, using that the Fourier transform $\mathcal{F}$
translates convolutions into multiplication, we get for any $g\in L_{\nu}^{1}(\mathbb{R}_{>0})$
and $u\in L_{\nu}^{2}(\mathbb{R})$ for some $\nu>\nu_{0}$ that 
\begin{align*}
g*u & =\sqrt{2\pi}\mathcal{L}_{\nu}^{*}\mathcal{L}_{\nu}g(m)\mathcal{L}_{\nu}u\\
 & =\sqrt{2\pi}\mathcal{L}_{\nu}^{*}\left(\mathcal{F}g\right)(m-i\nu)\mathcal{L}_{\nu}u\\
 & =\sqrt{2\pi}\mathcal{L}_{\nu}^{*}\left(\mathcal{F}g\right)\left(-i\frac{1}{\left(im+\nu\right)^{-1}}\right)\mathcal{L}_{\nu}u.
\end{align*}
The support and integrability condition of $g$ implies analyticity
of $M_{g}\coloneqq\sqrt{2\pi}\left(\mathcal{F}g\right)\left(-i\frac{1}{\left(\cdot\right)}\right)$
on $B(r,r)$ for $0<r<\frac{1}{2\nu_{0}}$. The computation also shows
that 
\begin{align*}
\left|g*u\right|_{\nu}^{2} & =\left|\sqrt{2\pi}\left(\mathcal{F}g\right)\left(-i\frac{1}{\left(im+\nu\right)^{-1}}\right)\mathcal{L}_{\nu}u\right|_{L^{2}}^{2}.\\
 & \leqq2\pi\left|\left(\mathcal{F}g\right)\left(-i\frac{1}{\left(i(\cdot)+\nu\right)^{-1}}\right)\right|_{\infty}^{2}\left|\mathcal{L}_{\nu}u\right|_{L^{2}}^{2}\\
 & \leqq2\pi\left|\left(\mathcal{F}g\right)\left((\cdot)-i\nu\right)\right|_{\infty}^{2}\left|u\right|_{\nu}^{2},
\end{align*}
where 
\begin{align*}
2\pi\left|\left(\mathcal{F}g\right)\left((\cdot)-i\nu\right)\right|_{\infty}^{2}\coloneqq\sup_{t\in\mathbb{R}}2\pi\left|\left(\mathcal{F}g\right)\left(t-i\nu\right)\right|^{2} & =\left|\int_{\mathbb{R}}e^{-i(t-i\nu)y}g(y)dy\right|^{2}\\
 & \leqq\left(\int_{\mathbb{R}}e^{-\nu y}\left|g(y)\right|dy\right)^{2},
\end{align*}
which tends to zero, if $\nu\to\infty.$ Thus, by our assumption on
the sequence $\left(g_{n}\right)_{n}$ having the uniform majorizing
function $h$, there exists $\nu_{1}>0$ such that we have 
\[
\epsilon\coloneqq\sup_{n\in\mathbb{N}}\|g_{n}*\|_{L(L_{\nu_{1}}^{2}(\mathbb{R}))}<1.
\]
Hence, we can reformulate (\ref{eq:Conv-eq}) as follows
\[
\left(1+M_{g_{n}}\left(\partial_{0}^{-1}\right)\right)u_{n}=f,
\]
Thus with $H_{0}=\{0\},$ $H_{1}=H$ and $N^{11}=\left(1+M_{g_{n}}\right)_{n}$
Theorem \ref{thm:translation-inv} is applicable. (Note that $\Re N_{n}^{11}\geqq1-\epsilon>0$
for all $n\in\mathbb{N}$). The assertion states that, for a suitable
subsequence for which we will use the same notation, we have 
\[
\left(1+M_{g_{n}}\left(\partial_{0}^{-1}\right)\right)\stackrel{G}{\longrightarrow}N_{hom,-1,11}\left(\partial_{0}^{-1}\right)^{-1}
\]
with 
\begin{align*}
N_{hom,-1,11}(z) & =\tau_{\textnormal{w}}\text{-}\lim_{n\to\infty}\left(1+M_{g_{n}}(z)\right)^{-1}\\
 & =\tau_{\textnormal{w}}\text{-}\lim_{n\to\infty}1+\sum_{\ell=1}^{\infty}M_{g_{n}}(z)^{\ell}\\
 & =\tau_{\textnormal{w}}\text{-}\lim_{n\to\infty}1+\sum_{\ell=1}^{\infty}M_{\left(g_{n}\right)^{*\ell}}(z)\\
 & =\tau_{\textnormal{w}}\text{-}\lim_{n\to\infty}1+M_{\sum_{\ell=1}^{\infty}\left(g_{n}\right)^{*\ell}}(z)
\end{align*}
for all $z\in B\left(\frac{1}{2\nu_{1}},\frac{1}{2\nu_{1}}\right)$
for some $\nu>\nu_{1}\geqq\nu_{0},$ where we denoted the $\ell$-fold
convolution with a function $g$ by $g^{*\ell}$, $\ell\in\mathbb{N}$.

\end{example}
In \cite{WauKal} we discussed the following variant of Example \ref{Ex:ODE_indep}.
\begin{example}
In the situation of Example \ref{Ex:ODE_indep}, we let $(h_{k})_{k}$
be a convergent sequence of positive real numbers with limit $h$.
Then Theorem \ref{thm:translation-inv} gives

{\footnotesize 
\begin{align*}
 & \partial_{0}a(k\cdot)+\tau_{-h_{k}}b(k\cdot)\\
 & \stackrel{G}{\longrightarrow}\partial_{0}\left(\int_{[0,1]^{n}}a(y)^{-1}dy\right)^{-1}\\
 & \quad\quad+\partial_{0}\sum_{k=1}^{\infty}\left(-\sum_{\ell=1}^{\infty}\left(\int_{[0,1]^{n}}a(y)^{-1}dy\right)^{-1}\int_{[0,1]^{n}}a(y)^{-1}\left(\tau_{-h}b(y)a(y)^{-1}\right)^{\ell}dy\left(-\partial_{0}^{-1}\right)^{\ell}\right)^{k}\left(\int_{[0,1]^{n}}a(y)^{-1}dy\right)^{-1}.
\end{align*}
}

Indeed, it suffices to observe that $\tau_{-h}=\mathcal{L}_{\nu}^{*}e^{-h(im+\nu)}\mathcal{L}_{\nu}$. 
\end{example}
Fractional differential equations are also admissible as the following
example shows.
\begin{example}
Again in the situation of Example \ref{Ex:ODE_indep}, let $\left(\alpha_{k}\right)_{k}$
and $\left(\beta_{k}\right)_{k}$ be convergent sequences in $]0,1]$
and $[-1,0]$ with limits $\alpha$ and $\beta$, resp. Then Theorem
\ref{thm:translation-inv} gives

{\footnotesize 
\begin{align*}
 & \partial_{0}^{\alpha_{k}}a(k\cdot)+\partial_{0}^{\beta_{k}}b(k\cdot)=\partial_{0}\partial_{0}^{\alpha_{k}-1}a(k\cdot)+\partial_{0}^{\beta_{k}}b(k\cdot)\\
 & \stackrel{G}{\longrightarrow}\partial_{0}^{\alpha}\left(\int_{[0,1]^{n}}a(y)^{-1}dy\right)^{-1}\\
 & \quad\quad+\partial_{0}^{\alpha}\sum_{k=1}^{\infty}\left(-\sum_{\ell=1}^{\infty}\partial_{0}^{\alpha-1}\left(\int_{[0,1]^{n}}a(y)^{-1}dy\right)^{-1}\int_{[0,1]^{n}}a(y)^{-1}\partial_{0}^{1-\alpha}\left(\partial_{0}^{1+\beta-\alpha}b(y)a(y)^{-1}\right)^{\ell}dy\left(-\partial_{0}^{-1}\right)^{\ell}\right)^{k}\\
 & \quad\quad\cdot\left(\int_{[0,1]^{n}}a(y)^{-1}dy\right)^{-1}.
\end{align*}
}\end{example}
\begin{rem}
Note that all the above theorems on homogenization of differential
equations straightforwardly apply to higher order equations. For example
the equation
\[
\sum_{k=0}^{n}\partial_{0}^{k}a_{k}u=f
\]
 can be reformulated as a first order system in the standard way.
Another way is to integrate $n-1$ times, to get that
\[
\sum_{k=0}^{n}\partial_{0}^{1+k-n}a_{k}u=\partial_{0}^{-(n-1)}f,
\]
 which is by setting $M(\partial_{0}^{-1})=a_{n}$ and $N(\partial_{0}^{-1})=\sum_{k=0}^{n-1}\partial_{0}^{1+k-n}a_{k}$
of the form treated in Theorem \ref{thm:translation-inv}.
\end{rem}

\section{Time-dependent coefficients\label{sec:Time-dependent-coefficients}}

In this section we treat operators depending on temporal and spatial
variables, which are, in contrast to the previous section, not time-translation
invariant. Thus, the structural hypothesis of being analytic functions
of $\partial_{0}^{-1}$ has to be lifted. Consequently, the expressions
for the limit equations do not simplify in the manner as they did
in the Theorems \ref{thm:time-independent} and \ref{thm:translation-inv}.
Particular ((non-)linear) equations have been considered in \cite{Mascarenhas1984,Tartar_Mem,Jiang,Jiang05,Petrini1998}.
The main objective of this section is to give a sufficient criterion
under which the choice of subsequences in Theorem \ref{thm:second-hom_thm}
is not required. We introduce the following notion.
\begin{defn}
Let $H$ be a Hilbert space. A family $((T_{n,\iota})_{n\in\mathbb{N}})_{\iota\in I}$
of sequences of linear operators in $L(H)$ is said to have the \emph{product-convergence
property}, if for all $k\in\mathbb{N}$ and $(\iota_{1},\ldots,\iota_{k})\in I^{k}$
the sequence $\left(\prod_{i=1}^{k}T_{n,\iota_{i}}\right)_{n}$ converges
in the weak operator topology of $L(H)$.\end{defn}
\begin{example}
\label{ex:per-mappings}Let $N,M\in\mathbb{N}$ and denote $\mathbb{P}\coloneqq\{a\colon\mathbb{R}^{N}\to\mathbb{C}^{M\times M};a\text{ is }[0,1]^{N}\text{-periodic}\}$.
Theorem \ref{thm:per} asserts that the family $\left(\left(a(k\cdot)\right)_{k\in\mathbb{N}}\right)_{a\in\mathbb{P}}$
has the product-convergence property in $L(L^{2}(\mathbb{R}^{N})^{M})$. 
\end{example}
We refer to the notion of homogenization algebras for other examples,
see e.g.\ \cite{Nguetseng1,Nguetseng2}. The main theorem of this
section reads as follows. Recall from Example \ref{time-dep} the
space $L_{s}^{\infty}(\mathbb{R};L(H))$ of strongly measurable bounded
functions with values in $L(H)$ endowed with the sup-norm. Moreover,
recall that for $A\in L_{s}^{\infty}(\mathbb{R};L(H))$ the associated
multiplication operator $A(m_{0})$ is evolutionary at $\nu$ for
every $\nu>0$.
\begin{thm}
\label{thm:product-convergence}Let $H$ be a Hilbert space, $\nu>0$.
Let $\left(\left(A_{\iota,n}\right)_{n}\right)_{\iota}$ be a family
of bounded sequences in $L_{s}^{\infty}(\mathbb{R};L(H))$. Assume
that the family $\left(\left(A_{\iota,n}(t)\right)_{n}\right)_{\iota,t\in\mathbb{R}}$
has the product-convergence property. Then $\left(\left(A_{\iota,n}(m_{0})\right)_{n},\left(\partial_{0}^{-1}\right)_{n}\right)_{\iota}$
has the product-convergence property.\end{thm}
\begin{rem}
(a) With the latter result, it is possible to deduce that the choice
of subsequences in Theorem \ref{thm:second-hom_thm} is not needed.
Indeed, assume that 
\[
\left(\begin{array}{cc}
\mathcal{M}_{n} & 0\\
0 & 0
\end{array}\right)+\left(\begin{array}{cc}
\mathcal{N}_{n}^{00} & \mathcal{N}_{n}^{01}\\
\mathcal{N}_{n}^{10} & \mathcal{N}_{n}^{11}
\end{array}\right)=\left(\begin{array}{cc}
M_{n}(m_{0}) & 0\\
0 & 0
\end{array}\right)+\left(\begin{array}{cc}
N_{n}^{00}(m_{0}) & N_{n}^{01}(m_{0})\\
N_{n}^{10}(m_{0}) & N_{n}^{11}(m_{0})
\end{array}\right)
\]
for some strongly measurable and bounded $M_{n}^{1},N_{n}^{00},N_{n}^{01},N_{n}^{10},N_{n}^{11}$
and assume that the family
\begin{align*}
 & \left(\left(\begin{array}{cc}
M_{n}(t) & 0\\
0 & 0
\end{array}\right)_{n},\left(\begin{array}{cc}
M_{n}(t)^{-1} & 0\\
0 & 0
\end{array}\right)_{n},\left(\begin{array}{cc}
N_{n}^{00}(t) & 0\\
0 & 0
\end{array}\right)_{n}\right.,\\
 & \quad\left.\left(\begin{array}{cc}
0 & N_{n}^{01}(t)\\
0 & 0
\end{array}\right)_{n},\left(\begin{array}{cc}
0 & 0\\
N_{n}^{10}(t) & 0
\end{array}\right)_{n},\left(\begin{array}{cc}
0 & 0\\
0 & N_{n}^{11}(t)
\end{array}\right)_{n},\left(\begin{array}{cc}
0 & 0\\
0 & N_{n}^{11}(t)^{-1}
\end{array}\right)_{n}\right)_{t\in\mathbb{R}}
\end{align*}
satisfies the product-convergence property. Then Theorem \ref{thm:product-convergence}
ensures that the limit expressions in Theorem \ref{thm:second-hom_thm}
converge without choosing subsequences. 

(b) The crucial fact in Theorem \ref{thm:product-convergence} is
that powers of $\partial_{0}^{-1}$ are involved. Indeed, let $H$
be a Hilbert space, $\nu>0$. Let $\left(\left(A_{\iota,n}\right)_{n}\right)_{\iota}$
be a family of bounded sequences in $L_{s}^{\infty}(\mathbb{R};L(H))$.
Assume that, for every $t\in\mathbb{R}$, the family $\left(\left(A_{\iota,n}(t)\right)_{n}\right)_{\iota}$
has the product-convergence property. Then $\left(\left(A_{\iota,n}(m_{0})\right)_{n}\right)_{\iota}$has
the product-convergence property. Showing the assertion for two sequences
$\left(A_{1,n}\right)_{n}$ and $\left(A_{2,n}\right)_{n}$ and using
the boundedness of the sequence $\left(A_{1,n}(m_{0})A_{2,n}(m_{0})\right)_{n}$,
we deduce that it suffices to show weak convergence on a dense subset.
For this to show let $K,L\subseteqq\mathbb{R}$ be bounded and measurable
and $\phi,\psi\in H$. We get for $n\in\mathbb{N}$ and $\nu>0$ that
\begin{align*}
\left\langle \chi_{K}\phi,A_{1,n}(m_{0})A_{2,n}(m_{0})\chi_{L}\psi\right\rangle _{\nu} & =\int_{K\cap L}\langle\phi,A_{1,n}(t)A_{2,n}(t)\psi\rangle e^{-2\nu t}dt\\
 & \to\int_{K\cap L}\lim_{n\to\infty}\langle\phi,A_{1,n}(t)A_{2,n}(t)\psi\rangle e^{-2\nu t}dt,
\end{align*}
by dominated convergence.\end{rem}
\begin{lem}
\label{lem:pointwise-integr} Let $H$ be a Hilbert space, $\nu>0$.
Let $\left(\left(A_{\iota,n}\right)_{n}\right)_{\iota\in I}$ be a
family of bounded sequences in $L_{s}^{\infty}(\mathbb{R};L(H))$.
Assume that the family $\left(\left(A_{\iota,n}(t)\right)_{n}\right)_{\iota,t\in\mathbb{R}}$
has the product-convergence property. Then%
\footnote{In what follows we adopt multiindex notation: For two operators $A$,
$B$ and $k=(k_{1},k_{2})\in\mathbb{N}_{0}^{2}$ we denote $(A,B)^{k}\coloneqq A^{k_{1}}B^{k_{2}}.$
If $k_{j}$ is a multiindex in $\mathbb{N}_{0}^{2}$, we denote its
first and second component respectively by $k_{j,1}$ and $k_{j,2}$.%
} $\prod_{j=1}^{k}\left(A_{\iota_{j},n}(m_{0}),\partial_{0}^{-1}\right)^{\ell_{j}}$
converges in the weak operator toplogogy for all $k\in\mathbb{N}$,
$\ell_{1},\ldots,\ell_{k}\in\{0,1\}\times\mathbb{N}$ and $\iota_{1},\ldots,\iota_{k}\in I$.\end{lem}
\begin{proof}
Let $k\in\mathbb{N}$, $\ell_{1},\ldots,\ell_{k}\in\{0,1\}\times\mathbb{N}$
and $\iota_{1},\ldots,\iota_{k}\in I$. Moreover, take $\phi,\psi\in H$
and $K,L\subseteqq\mathbb{R}$ be bounded and measurable. For $n\in\mathbb{N}$
and $\nu>0$ we compute
\begin{align*}
 & \left\langle \chi_{K}\phi,\prod_{j=1}^{k}\left(A_{\iota_{j},n}(m_{0}),\partial_{0}^{-1}\right)^{\ell_{j}}\chi_{L}\psi\right\rangle _{\nu}\\
 & =\int_{K}\left\langle \phi,A_{\iota_{1},n}(s_{0}^{0})^{\ell_{1,1}}\int_{-\infty}^{s_{0}^{0}}\int_{-\infty}^{s_{\ell_{1,2}-1}^{1}}\cdots\int_{-\infty}^{s_{1}^{1}}A_{\iota_{2},n}(s_{0}^{1})^{\ell_{2,1}}\int_{-\infty}^{s_{0}^{1}}\int_{-\infty}^{s_{\ell_{2,2}-1}^{2}}\cdots\int_{-\infty}^{s_{1}^{2}}\cdots\right.\\
 & \quad\left.A_{\iota_{k},n}(s_{0}^{k-1})^{\ell_{k,1}}\int_{-\infty}^{s_{0}^{k-1}}\int_{-\infty}^{s_{\ell_{k,2}-1}^{k}}\cdots\int_{-\infty}^{s_{1}^{k}}\chi_{L}(s_{0}^{k})\psi\right\rangle \\
 & \quad\quad ds_{0}^{k}\cdots ds_{\ell_{k,2}-2}^{k}ds_{\ell_{k,2}-1}^{k}\cdots ds_{0}^{2}\cdots ds_{\ell_{2,2}-2}^{2}ds_{\ell_{2,2}-1}^{2}ds_{0}^{1}\cdots ds_{\ell_{1,2}-2}^{1}ds_{\ell_{1,2}-1}^{1}e^{-2\nu s_{0}^{0}}ds_{0}^{0}\\
 & =\int_{K}\int_{-\infty}^{s_{0}^{0}}\int_{-\infty}^{s_{\ell_{1,2}-1}^{1}}\cdots\int_{-\infty}^{s_{1}^{1}}\int_{-\infty}^{s_{0}^{1}}\int_{-\infty}^{s_{\ell_{2,2}-1}^{2}}\cdots\int_{-\infty}^{s_{1}^{2}}\cdots\int_{-\infty}^{s_{0}^{k-1}}\int_{-\infty}^{s_{\ell_{k,2}-1}^{k}}\cdots\int_{-\infty}^{s_{1}^{k}}\\
 & \quad\quad\left\langle \phi,A_{\iota_{1},n}(t)^{\ell_{1,1}}A_{\iota_{2},n}(s_{0})^{\ell_{2,1}}\cdots\chi_{L}(s_{0}^{k})\psi\right\rangle \\
 & \quad\quad ds_{0}^{k}\cdots ds_{\ell_{k,2}-2}^{k}ds_{\ell_{k,2}-1}^{k}\cdots ds_{0}^{2}\cdots ds_{\ell_{2,2}-2}^{2}ds_{\ell_{2,2}-1}^{2}ds_{0}^{1}\cdots ds_{\ell_{1,2}-2}^{1}ds_{\ell_{1,2}-1}^{1}e^{-2\nu s_{0}^{0}}ds_{0}^{0}.
\end{align*}
Using dominated convergence, we deduce the convergence of the latter
expression. 
\end{proof}
~
\begin{proof}[Proof of Theorem \ref{thm:product-convergence}]
 The proof follows easily with Lemma \ref{lem:pointwise-integr}.
\end{proof}
Theorem \ref{thm:product-convergence} serves as a possibility to
deduce $G$-convergence of differential operators, where the coefficients
take values in, for example, periodic mappings as in Example \ref{ex:per-mappings}.
Another instance is given in the following example.
\begin{example}
\label{periodic-in-time}Let $A,B\in L^{\infty}(\mathbb{R})$ be $1$-periodic,
$f\in C_{\infty,c}(\mathbb{R})$. Assume that $A\geqq c$ for some
$c>0$. For $n\in\mathbb{N}$ and $\nu>0$ consider 
\[
\left(\partial_{0}A(n\cdot m_{0})+B(n\cdot m_{0})\right)u_{n}=f.
\]
Recall that from Theorem \ref{thm:first-hom_thm}, in order to compute
the limit equation, we have to compute expressions of the form 
\[
\mathcal{M}_{hom,\ell}=\tau_{\textnormal{w}}\textnormal{-}\lim_{n\to\infty}\mathcal{M}_{n}^{-1}\left(-\partial_{0}^{-1}\mathcal{N}_{n}\mathcal{M}_{n}^{-1}\right)^{\ell},
\]
where $\mathcal{M}_{n}=A(n\cdot m_{0})$ and $\mathcal{N}_{n}=B(n\cdot m_{0})$,
$\ell\in\mathbb{N}$.
\end{example}
In order to deduce $G$-convergence in the latter example we need
the following theorem.
\begin{thm}
\label{thm:Time-product}Let $A_{1},\ldots,A_{k}\in L^{\infty}(\mathbb{R})$
be $1$-periodic. Then for every $\nu>0$ we have 
\[
\mathcal{A}_{n}\coloneqq A_{1}(n\cdot m_{0})\left(\prod_{j=1}^{k-1}\partial_{0}^{-1}A_{j+1}(n\cdot m_{0})\right)\stackrel{\tau_{\textnormal{w}},n\to\infty}{-\!\!\!-\!\!\!-\!\!\!\!\longrightarrow}\left(\partial_{0}^{-1}\right)^{k-1}\prod_{j=1}^{k}\int_{0}^{1}A_{j}(y)dy\in L\left(L_{\nu}^{2}(\mathbb{R})\right).
\]
\end{thm}
\begin{proof}
For $n\in\mathbb{N}$ and $K,L\subseteqq\mathbb{R}$ bounded, measurable
we compute
\begin{align*}
\langle\chi_{K},\mathcal{A}_{n}\chi_{L}\rangle_{\nu} & =\int_{K}A_{1}(nt_{1})\int_{-\infty}^{t_{1}}A_{2}(nt_{2})\int_{-\infty}^{t_{2}}\cdots\int_{-\infty}^{t_{k-1}}A_{k}(nt_{k})\chi_{L}(t_{k})dt_{k}\cdots dt_{2}e^{-2\nu t_{1}}dt_{1}\\
 & =\int_{K}\int_{-\infty}^{t_{1}}\int_{-\infty}^{t_{2}}\cdots\int_{-\infty}^{t_{k-1}}\left(\prod_{j=1}^{k}A_{j}(nt_{j})\right)\chi_{L}(t_{k})e^{-2\nu t_{1}}dt_{k}\cdots dt_{1}\\
 & =\underbrace{\int_{\mathbb{R}}\cdots\int_{\mathbb{R}}}_{k\text{-times}}\left(\prod_{j=1}^{k}A_{j}(nt_{j})\right)\chi_{K}(t_{1})\left(\prod_{j=2}^{k}\chi_{\mathbb{R}_{>0}}(t_{j-1}-t_{j})\right)\chi_{L}(t_{k})e^{-2\nu t_{1}}dt_{k}\cdots dt_{1}.
\end{align*}
Now, observe that $(t_{1},\cdots,t_{k})\mapsto\chi_{K}(t_{1})\left(\prod_{j=2}^{k}\chi_{\mathbb{R}_{>0}}(t_{j-1}-t_{j})\right)\chi_{L}(t_{k})e^{-2\nu t_{1}}\in L^{1}(\mathbb{R}^{k}).$
Moreover, the mapping $(t_{1},\cdots,t_{k})\mapsto\prod_{j=1}^{k}A_{j}(t_{j})$
is $[0,1]^{k}$-periodic. Thus, by Theorem \ref{thm:per}, we conclude
that 
\[
\langle\chi_{K},\mathcal{A}_{n}\chi_{L}\rangle_{\nu}\to\left\langle \chi_{K},\left(\partial_{0}^{-1}\right)^{k-1}\prod_{j=1}^{k}\int_{0}^{1}A_{j}(y)dy\chi_{L}\right\rangle 
\]
 as $n\to\infty$ for all $K,L\subseteqq\mathbb{R}$ bounded and measurable.
A density argument concludes the proof. \end{proof}
\begin{example}[Example \ref{periodic-in-time} continued]
 Thus, with the Theorems \ref{thm:Time-product} and \ref{thm:first-hom_thm},
we conclude that $\left(\partial_{0}A(n\cdot m_{0})+B(n\cdot m_{0})\right)$
$G$-converges to 
\begin{align*}
 & \partial_{0}\left(\int_{0}^{1}\frac{1}{A(y)}dy\right)^{-1}\\
 & \quad+\partial_{0}\sum_{k=1}^{\infty}\left(-\sum_{\ell=1}^{\infty}\left(\int_{0}^{1}\frac{1}{A(y)}dy\right)^{-1}\int_{0}^{1}\frac{1}{A(y)}dy\left(-\partial_{0}^{-1}\int_{0}^{1}\frac{B(y)}{A(y)}dy\right)^{\ell}\right)^{k}\left(\int_{0}^{1}\frac{1}{A(y)}dy\right)^{-1}\\
 & =\partial_{0}\left(\int_{0}^{1}\frac{1}{A(y)}dy\right)^{-1}\left(1+\sum_{k=1}^{\infty}\left(-\sum_{\ell=1}^{\infty}\left(-\partial_{0}^{-1}\int_{0}^{1}\frac{B(y)}{A(y)}dy\right)^{\ell}\right)^{k}\right)\\
 & =\partial_{0}\left(\int_{0}^{1}\frac{1}{A(y)}dy\right)^{-1}\sum_{k=0}^{\infty}\left(-\sum_{\ell=1}^{\infty}\left(-\partial_{0}^{-1}\int_{0}^{1}\frac{B(y)}{A(y)}dy\right)^{\ell}\right)^{k}\\
 & =\partial_{0}\left(\int_{0}^{1}\frac{1}{A(y)}dy\right)^{-1}\left(1+\sum_{\ell=1}^{\infty}\left(-\partial_{0}^{-1}\int_{0}^{1}\frac{B(y)}{A(y)}dy\right)^{\ell}\right)^{-1}\\
 & =\partial_{0}\left(\int_{0}^{1}\frac{1}{A(y)}dy\right)^{-1}\left(\sum_{\ell=0}^{\infty}\left(-\partial_{0}^{-1}\int_{0}^{1}\frac{B(y)}{A(y)}dy\right)^{\ell}\right)^{-1}\\
 & =\partial_{0}\left(\int_{0}^{1}\frac{1}{A(y)}dy\right)^{-1}\left(1+\partial_{0}^{-1}\int_{0}^{1}\frac{B(y)}{A(y)}dy\right)\\
 & =\partial_{0}\left(\int_{0}^{1}\frac{1}{A(y)}dy\right)^{-1}+\left(\int_{0}^{1}\frac{1}{A(y)}dy\right)^{-1}\int_{0}^{1}\frac{B(y)}{A(y)}dy.
\end{align*}
\end{example}
\begin{rem}
In \cite{Petrini1998}, the authors consider an equation of the form
$\left(\partial_{0}+a_{n}(m_{0})\right)u_{n}=f$ in the space $L^{2}(\mathbb{R};L^{2}(\mathbb{R}))$
with $\left(a_{n}\right)_{n}$ being a bounded sequence in $L^{\infty}(\mathbb{R}\times\mathbb{R})$.
Assuming weak-$*$-convergence of $\left(a_{n}\right)_{n}$ , the
author shows weak convergence of $(u_{n})_{n}$. The limit equation
is a convolution equation involving the Young-measure associated to
the sequence $\left(a_{n}\right)_{n}$. Within our reasoning, we cannot
show that the whole sequence converges, unless any power of $\left(a_{n}\right)_{n}$
converges in the weak-$*$ topology of $L^{\infty}$. However, as
we illustrated above (see e.g.\ Example \ref{Drude-Born-Fedorov-Ex})
our approach has a wide range of applications, where the method involving
Young-measures might fail to work.
\end{rem}

\section{Proof of the main theorems\label{sec:Proof-of-the}}

We will finally prove our main theorems. The proof relies on elementary
Hilbert space concepts. We emphasize that the generality of the perspective
hardly allows the introduction of Young-measures, which have proven
to be useful in particular cases (see the sections above for a detailed
discussion). Before we give a detailed account of the proofs of our
main theorems, we state the following auxilaury result, which we state
without proof.
\begin{lem}
\label{lem:pos-def} Let $H$ be a Hilbert space, $T\in L(H)$. Assume
that $\Re T\geqq c$ for some $c>0$. Then $\|T^{-1}\|\leqq\frac{1}{c}$
and $\Re T^{-1}\geqq\frac{c}{\|T\|^{2}}.$\end{lem}
\begin{proof}[Proof of Theorem \ref{thm:first-hom_thm}]
 For $f\in C_{\infty,c}(\mathbb{R};H)$ let $u_{n}$ solve
\[
\left(\partial_{0}\mathcal{M}_{n}+\mathcal{N}_{n}\right)u_{n}=f.
\]
This yields
\begin{align*}
u_{n} & =\mathcal{M}_{n}^{-1}\left(1+\partial_{0}^{-1}\mathcal{N}_{n}\mathcal{M}_{n}^{-1}\right)^{-1}\partial_{0}^{-1}f\\
 & =\mathcal{M}_{n}^{-1}\sum_{\ell=0}^{\infty}\left(-\partial_{0}^{-1}\mathcal{N}_{n}\mathcal{M}_{n}^{-1}\right)^{\ell}\partial_{0}^{-1}f\\
 & =\left(\mathcal{M}_{n}^{-1}+\sum_{\ell=1}^{\infty}\mathcal{M}_{n}^{-1}\left(-\partial_{0}^{-1}\mathcal{N}_{n}\mathcal{M}_{n}^{-1}\right)^{\ell}\right)\partial_{0}^{-1}f.
\end{align*}
Hence, choosing an appropriate subsequence, we arrive at an expression
of the form
\[
u=\left(\mathcal{M}_{hom,0}+\sum_{\ell=1}^{\infty}\mathcal{M}_{hom,\ell}\right)\partial_{0}^{-1}f.
\]
We remark here that due to the (standard) estimate $\|T\|\leqq\liminf_{k\to\infty}\|T_{k}\|$
for a sequence $\left(T_{k}\right)_{k}$ of bounded linear operators
in some Hilbert space converging to $T$, the series $\sum_{\ell=1}^{\infty}\mathcal{M}_{hom,\ell}$
converges with respect to the operator norm if $\nu$ is chosen large
enough. Using the positive definiteness of $\mathcal{M}_{n}$ for
all $n\in\mathbb{N}$ and Lemma \ref{lem:pos-def}, we deduce that
\[
\Re\mathcal{M}_{n}^{-1}\geqq\frac{c}{\sup_{n\in\mathbb{N}}\|\mathcal{M}_{n}\|^{2}}.
\]
By $\|\mathcal{M}_{n}^{-1}\|\leqq\frac{1}{c},$ we conclude that 
\[
\Re\mathcal{M}_{hom,0}\geqq\frac{c}{\sup_{n\in\mathbb{N}}\|\mathcal{M}_{n}\|^{2}}
\]
 and
\[
\Re\mathcal{M}_{hom,0}^{-1}\geqq\frac{c^{3}}{\sup_{n\in\mathbb{N}}\|\mathcal{M}_{n}\|^{2}}.
\]
We arrive at
\begin{align*}
f & =\partial_{0}\left(1+\sum_{\ell=1}^{\infty}\mathcal{M}_{hom,0}^{-1}\mathcal{M}_{hom,\ell}\right)^{-1}\mathcal{M}_{hom,0}^{-1}u\\
 & =\partial_{0}\sum_{k=0}^{\infty}\left(-\sum_{\ell=1}^{\infty}\mathcal{M}_{hom,0}^{-1}\mathcal{M}_{hom,\ell}\right)^{k}\mathcal{M}_{hom,0}^{-1}u\\
 & =\partial_{0}\left(1+\sum_{k=1}^{\infty}\left(-\sum_{\ell=1}^{\infty}\mathcal{M}_{hom,0}^{-1}\mathcal{M}_{hom,\ell}\right)^{k}\right)\mathcal{M}_{hom,0}^{-1}u\\
 & =\partial_{0}\mathcal{M}_{hom,0}^{-1}u+\partial_{0}\sum_{k=1}^{\infty}\left(-\sum_{\ell=1}^{\infty}\mathcal{M}_{hom,0}^{-1}\mathcal{M}_{hom,\ell}\right)^{k}\mathcal{M}_{hom,0}^{-1}u.\tag*{{\qedhere}}
\end{align*}

\end{proof}
~
\begin{proof}[Proof of Theorem \ref{thm:second-hom_thm}]
 We observe 
\begin{align*}
 & \left(\begin{array}{cc}
\partial_{0}\mathcal{M}_{n}+\mathcal{N}_{n}^{00} & \mathcal{N}_{n}^{01}\\
\mathcal{N}_{n}^{10} & \mathcal{N}_{n}^{11}
\end{array}\right)\\
 & =\left(\begin{array}{cc}
1 & \quad\mathcal{N}_{n}^{01}\left(\mathcal{N}_{n}^{11}\right)^{-1}\\
0 & 1
\end{array}\right)\left(\begin{array}{cc}
\partial_{0}\mathcal{M}_{n}+\mathcal{N}_{n}^{00}-\mathcal{N}_{n}^{01}\left(\mathcal{N}_{n}^{11}\right)^{-1}\mathcal{N}_{n}^{10} & 0\\
0 & \mathcal{N}_{n}^{11}
\end{array}\right)\left(\begin{array}{cc}
1 & 0\\
\left(\mathcal{N}_{n}^{11}\right)^{-1}\mathcal{N}_{n}^{10}\quad & 1
\end{array}\right).
\end{align*}
Thus, with $B:=\left(\partial_{0}\mathcal{M}_{n}+\mathcal{N}_{n}^{00}-\mathcal{N}_{n}^{01}\left(\mathcal{N}_{n}^{11}\right)^{-1}\mathcal{N}_{n}^{10}\right)^{-1}$
{\footnotesize
\begin{align*}
 & \left(\begin{array}{cc}
\partial_{0}\mathcal{M}_{n}+\mathcal{N}_{n}^{00} & \mathcal{N}_{n}^{01}\\
\mathcal{N}_{n}^{10} & \mathcal{N}_{n}^{11}
\end{array}\right)^{-1}\\
 & =\left(\begin{array}{cc}
1 & 0\\
-\left(\mathcal{N}_{n}^{11}\right)^{-1}\mathcal{N}_{n}^{10} & 1
\end{array}\right)\left(\begin{array}{cc}
\left(\partial_{0}\mathcal{M}_{n}+\mathcal{N}_{n}^{00}-\mathcal{N}_{n}^{01}\left(\mathcal{N}_{n}^{11}\right)^{-1}\mathcal{N}_{n}^{10}\right)^{-1} & 0\\
0 & \left(\mathcal{N}_{n}^{11}\right)^{-1}
\end{array}\right)\left(\begin{array}{cc}
1 & -\mathcal{N}_{n}^{01}\left(\mathcal{N}_{n}^{11}\right)^{-1}\\
0 & 1
\end{array}\right)\\
 & =\left(\begin{array}{cc}
B & 0\\
-\left(\mathcal{N}_{n}^{11}\right)^{-1}\mathcal{N}_{n}^{10}B\quad & \quad\left(\mathcal{N}_{n}^{11}\right)^{-1}
\end{array}\right)\left(\begin{array}{cc}
1 & -\mathcal{N}_{n}^{01}\left(\mathcal{N}_{n}^{11}\right)^{-1}\\
0 & 1
\end{array}\right)\\
 & =\left(\begin{array}{cc}
B & -B\mathcal{N}_{n}^{01}\left(\mathcal{N}_{n}^{11}\right)^{-1}\\
-\left(\mathcal{N}_{n}^{11}\right)^{-1}\mathcal{N}_{n}^{10}B\quad & \quad\left(\mathcal{N}_{n}^{11}\right)^{-1}\mathcal{N}_{n}^{10}B\mathcal{N}_{n}^{01}\left(\mathcal{N}_{n}^{11}\right)^{-1}+\left(\mathcal{N}_{n}^{11}\right)^{-1}
\end{array}\right).
\end{align*}

}

With the Neumann series expression derived in the previous theorem,
i.e., 
\[
B=\mathcal{M}_{n}^{-1}\partial_{0}^{-1}+\sum_{\ell=1}^{\infty}\mathcal{M}_{n}^{-1}\left(-\partial_{0}^{-1}\mathcal{N}_{n}\mathcal{M}_{n}^{-1}\right)^{\ell}\partial_{0}^{-1}
\]
with $\mathcal{N}_{n}=\mathcal{N}_{n}^{00}-\mathcal{N}_{n}^{01}\left(\mathcal{N}_{n}^{11}\right)^{-1}\mathcal{N}_{n}^{10}$,
we get that{\footnotesize
\begin{align*}
 & \left(\begin{array}{cc}
\partial_{0}\mathcal{M}_{n}+\mathcal{N}_{n}^{00} & \mathcal{N}_{n}^{01}\\
\mathcal{N}_{n}^{10} & \mathcal{N}_{n}^{11}
\end{array}\right)^{-1}\\
 & =\sum_{\ell=0}^{\infty}\left(\begin{array}{cc}
\mathcal{M}_{n}^{-1}\left(-\partial_{0}^{-1}\mathcal{N}_{n}\mathcal{M}_{n}^{-1}\right)^{\ell}\partial_{0}^{-1} & -\mathcal{M}_{n}^{-1}\left(-\partial_{0}^{-1}\mathcal{N}_{n}\mathcal{M}_{n}^{-1}\right)^{\ell}\partial_{0}^{-1}\mathcal{N}_{n}^{01}\left(\mathcal{N}_{n}^{11}\right)^{-1}\\
-\left(\mathcal{N}_{n}^{11}\right)^{-1}\mathcal{N}_{n}^{10}\mathcal{M}_{n}^{-1}\left(-\partial_{0}^{-1}\mathcal{N}_{n}\mathcal{M}_{n}^{-1}\right)^{\ell}\partial_{0}^{-1}\quad & \quad\left(\mathcal{N}_{n}^{11}\right)^{-1}\mathcal{N}_{n}^{10}\mathcal{M}_{n}^{-1}\left(-\partial_{0}^{-1}\mathcal{N}_{n}\mathcal{M}_{n}^{-1}\right)^{\ell}\partial_{0}^{-1}\mathcal{N}_{n}^{01}\left(\mathcal{N}_{n}^{11}\right)^{-1}
\end{array}\right)\\
 & \quad+\left(\begin{array}{cc}
0 & 0\\
0 & \left(\mathcal{N}_{n}^{11}\right)^{-1}
\end{array}\right).
\end{align*}
}With Theorem \ref{thm:first-hom_thm}, we deduce convergence of
the top left corner in the latter matrix. Similarly, we deduce convergence
of the other expressions. Thus, for a suitable choice of subsequences,
we arrive at
\[
\sum_{\ell=1}^{\infty}\left(\begin{array}{cc}
\mathcal{M}_{hom,\ell,00}\partial_{0}^{-1} & \mathcal{M}_{hom,\ell,01}\\
\mathcal{M}_{hom,\ell,10}\partial_{0}^{-1} & \mathcal{M}_{hom,\ell,11}
\end{array}\right)+\left(\begin{array}{cc}
\mathcal{M}_{hom,0,00}\partial_{0}^{-1} & \mathcal{M}_{hom,0,01}\\
\mathcal{M}_{hom,0,10}\partial_{0}^{-1} & \mathcal{M}_{hom,0,11}+\mathcal{N}_{hom,-1,11}
\end{array}\right).
\]
We observe that 
\begin{align*}
 & \sum_{\ell=1}^{\infty}\left(\begin{array}{cc}
\mathcal{M}_{hom,\ell,00}\partial_{0}^{-1} & \mathcal{M}_{hom,\ell,01}\\
\mathcal{M}_{hom,\ell,10}\partial_{0}^{-1} & \mathcal{M}_{hom,\ell,11}
\end{array}\right)+\left(\begin{array}{cc}
\mathcal{M}_{hom,0,00}\partial_{0}^{-1} & \mathcal{M}_{hom,0,01}\\
\mathcal{M}_{hom,0,10}\partial_{0}^{-1} & \mathcal{M}_{hom,0,11}+\mathcal{N}_{hom,-1,11}
\end{array}\right)\\
 & =\left(\mathcal{M}^{(1)}+\left(\begin{array}{cc}
\mathcal{M}_{hom,0,00} & 0\\
0 & \mathcal{N}_{hom,-1,11}
\end{array}\right)\right)\left(\begin{array}{cc}
\partial_{0}^{-1} & 0\\
0 & 1
\end{array}\right).
\end{align*}
Moreover, note that the operator $\mathcal{M}^{(1)}$ has norm arbitrarily
small if $\nu$ was chosen large enough. Hence, the operator 
\begin{align*}
 & \left(\mathcal{M}^{(1)}+\left(\begin{array}{cc}
\mathcal{M}_{hom,0,00} & 0\\
0 & \mathcal{N}_{hom,-1,11}
\end{array}\right)\right)\\
 & =\left(\begin{array}{cc}
\mathcal{M}_{hom,0,00} & 0\\
0 & \mathcal{N}_{hom,-1,11}
\end{array}\right)\left(\left(\begin{array}{cc}
\mathcal{M}_{hom,0,00} & 0\\
0 & \mathcal{N}_{hom,-1,11}
\end{array}\right)^{-1}\mathcal{M}^{(1)}+1\right)
\end{align*}
is invertible. This gives

\begin{align*}
 & \left(\left(\mathcal{M}^{(1)}+\left(\begin{array}{cc}
\mathcal{M}_{hom,0,00} & 0\\
0 & \mathcal{N}_{hom,-1,11}
\end{array}\right)\right)\left(\begin{array}{cc}
\partial_{0}^{-1} & 0\\
0 & 1
\end{array}\right)\right)^{-1}\\
 & =\left(\begin{array}{cc}
\partial_{0} & 0\\
0 & 1
\end{array}\right)\sum_{\ell=0}^{\infty}\left(-\left(\begin{array}{cc}
\mathcal{M}_{hom,0,00} & 0\\
0 & \mathcal{N}_{hom,-1,11}
\end{array}\right)^{-1}\mathcal{M}^{(1)}\right)^{\ell}\left(\begin{array}{cc}
\mathcal{M}_{hom,0,00}^{-1} & 0\\
0 & \mathcal{N}_{hom,-1,11}^{-1}
\end{array}\right)\\
 & =\left(\begin{array}{cc}
\partial_{0} & 0\\
0 & 1
\end{array}\right)\left(\left(\begin{array}{cc}
\mathcal{M}_{hom,0,00}^{-1} & 0\\
0 & \mathcal{N}_{hom,-1,11}^{-1}
\end{array}\right)^{\phantom{\ell}}\right.\\
 & \quad+\left.\sum_{\ell=1}^{\infty}\left(-\left(\begin{array}{cc}
\mathcal{M}_{hom,0,00} & 0\\
0 & \mathcal{N}_{hom,-1,11}
\end{array}\right)^{-1}\mathcal{M}^{(1)}\right)^{\ell}\left(\begin{array}{cc}
\mathcal{M}_{hom,0,00}^{-1} & 0\\
0 & \mathcal{N}_{hom,-1,11}^{-1}
\end{array}\right)\right).\tag*{{\qedhere}}
\end{align*}
\end{proof}


\begin{thebibliography}{10}
\bibitem{Anton}N. Antoni\'c.  \newblock {Memory effects in homogenisation:
Linear second-order equations.} \newblock {\em Arch. Ration. Mech.
Anal. } 125(1): 1--24, 1993.

\bibitem{Ball}J.M. Ball. \newblock {A Version of the fundamental
Theorem for Young Measures} \newblock {\em PDE's and Continuum Models
of Phase Transition, Lecture Notes in Physics} 344:207--215, 1989.

\bibitem{BenLiPap}A. Bensoussan, J.L. Lions, and G. Papanicolaou.
\newblock {\em {Asymptotic Analysis for Periodic Structures}}. \newblock
North-Holland, Amsterdam, 1978.

\bibitem{CioDon}D. Cioranescu and P. Donato. \newblock {\em {An
Introduction to Homogenization}}. \newblock Oxford University Press,
New York, 2010.

\bibitem{Conway}J.B. Conway. \newblock {\em { A course in functional
analysis.}}. \newblock 2nd edition. In Graduate Texts inMathematics,
96, Vol. xvi. Springer-Verlag: New York etc. p. 399, 1990.

\bibitem{Filonov}N. Filonov. \newblock {Spectral analysis of the
selfadjoint operator curl in a region of finite measure.} \newblock{ \em 
St. Petersburg Math. J.}, 11(6):1085\textendash{}1095, 2000.

\bibitem{Picard-Freymond}H. Freymond and R. Picard. \newblock{ On
electromagnetic waves in complex linear media in nonsmooth domains} \newblock{\em Mathematical
Methods in the Applied Sciences } To appear. 2012.

\bibitem{Kalauch}A. Kalauch, R. Picard, S. Siegmund, S. Trostorff,
and M. Waurick. A Hilbert Space Perspective on Ordinary Differential
Equations with Memory Term. Technical Report, TU Dresden, MATH-AN-015-2012,
submitted. (\url{http://arXiv:1204.2924})

\bibitem{JacZwart1}B.~Jacob and H.~J.~Zwart. \newblock {Linear
Port-Hamiltonian systems on infinite-dimensional spaces. } \newblock
Operator Theory: Advances and Applications 223. Basel: Birkh\"auser,
2012. 

\bibitem{Jiang}J.-S. Jiang, K.-H. Kuo, and C.-K. Lin \newblock {
On the Homogenization of Second Order Differential Equations } \newblock{ \em Taiwanese
Journal of Mathematics} 9: 215--236, 2005.

\bibitem{Jiang05}J.-S. Jiang, K.-H. Kuo, and C.-K. Lin \newblock {
Homogenization of second order equation with spatial dependent coefficient.} \newblock{ \em 
Discrete Contin. Dyn. Syst. } 12(2): 303--313, 2005.

\bibitem{Nguetseng1}G. Nguetseng. \newblock {Homogenization structures
and applications. I.} \newblock {\em { Z. Anal. Anwend.}}, 22(1):73\textendash{}107,
2003. 

\bibitem{Nguetseng2}G. Nguetseng. \newblock {Homogenization structures
and applications. II. } \newblock {\em {Z. Anal. Anwend.}}, 23(3):483\textendash{}508,
2004.

\bibitem{Mascarenhas1984}M.-L. Mascarenhas \newblock { A linear
homogenization problem with time dependent coefficient.} \newblock {\em {Trans.
Am. Math. Soc.}} 281: 179--195, 1984.

\bibitem{Petrini1998}M. Petrini. \newblock { Homogenization of linear
and nonlinear ordinary differential equations with time depending
coefficients.} \newblock {\em {Rend. Semin. Mat. Univ. Padova }}99:
133--159, 1998. 

\bibitem{Picard}R. Picard and D. McGhee. \newblock {\em {Partial
Differential Equations: A unified Hilbert Space Approach}}, volume
55 of {\em Expositions in Mathematics.} \newblock DeGruyter, Berlin,
2011.

\bibitem{Picard_curl_ext}R. Picard.\newblock { On a selfadjoint
realization of curl in exterior domains. } \newblock {\em {Math.
Z.}}, 229(2):319\textendash{}338, 1998.

\bibitem{Piccinini77}L.C. Piccinini \newblock { G-convergence for
ordinary differential equations with Peano phaenomenon.} \newblock {\em {Rend.
Sem. Mat. Univ. Padova}} 58: 65--86, 1977.

\bibitem{Piccinini1978}L.C. Piccinini \newblock {Homogeneization
problems for ordinary differential equations.} \newblock {\em {Rend.
Circ. Mat. Palermo, II. Ser.}} 27: 95--112, 1978.

\bibitem{Tartar_Mem} L. Tartar.  \newblock { Memory effects and
homogenization. } \newblock {\em Archive for Rational Mechanics and
Analysis} 111:121--133, 1990.

\bibitem{TartarNonlHom}L. Tartar.  \newblock { Nonlocal effects
induced by homogenization} \newblock {\em Partial Differential Equations
and the Calculus of Variations, Essays in Honor of Ennio De Giorgi} 2:925--938,
1989.

\bibitem{Waurick2011Diss} M. Waurick. \newblock {\em {Limiting Processes
in Evolutionary Equations - A Hilbert Space Approach to Homogenization.}}
\newblock Dissertation, TU Dresden, \url{http://nbn-resolving.de/urn:nbn:de:bsz:14-qucosa-67442},
2011.

\bibitem{WaurickMMA} M. Waurick. \newblock {A Hilbert space approach
to homogenization of linear ordinary differential equations including
delay and memory terms.} \newblock {\em Math. Methods Appl. Sci.},
35(9): 1067--1077, 2012.

\bibitem{Waurick2012Asy}M. Waurick. \newblock Homogenization of
a class of linear partial differential equations. \newblock {\em
Asymptotic Analysis}, 2012. In press.

\bibitem{Waurick2012aDIE}M. Waurick. \newblock How far away is the
harmonic mean from the homogenized matrix? \newblock Technical report,
TU Dresden, MATH-AN-07-2012, submitted. (\url{arXiv:1204.3768v3}).

\bibitem{Waurick2013frac}M. Waurick. \newblock Homogenization in
fractional elasticity \newblock Technical report, TU Dresden, MATH-AN-02-2013,
submitted. (\url{arxiv:1302.1731}).

\bibitem{WauKal}M. Waurick and M. Kaliske. \newblock A Note on Homogenization
of Ordinary Differential Equations with Delay Term. \newblock {\em 
PAMM } 11, 889-890, 2011.

\bibitem{Gcon1}V.V. Zhikov, S.M. Kozlov, O.A. Oleinik, and K. T'en
Ngoan. \newblock {Averaging and G-convergence of Differential Operators}. \newblock { \em 
Russian Mathematical Surveys}, 34:69--147, 1979.\end{thebibliography}
\end{document}